\documentclass[a4paper,reqno,final]{amsart}

\usepackage[foot]{amsaddr}
\usepackage{graphicx,enumerate,nicefrac,color,bm,cite}
\usepackage{amsmath, amssymb, amstext, amsfonts}
\usepackage{stmaryrd}
\usepackage[dvipsnames]{xcolor}
\usepackage{geometry}
\usepackage{hyperref}
\usepackage{subfig}
\usepackage{a4wide}
\usepackage{algorithm,algpseudocode,tabularx}

\makeatletter
\newcommand{\multiline}[1]{%
  \begin{tabularx}{\dimexpr\linewidth-\ALG@thistlm}[t]{@{}X@{}}
    #1
  \end{tabularx}
}
\makeatother

\newcommand{\norm}[1]{\left\|#1\right\|}

\newcommand{\operator}[1]{\mathsf{#1}}
\newcommand{\A}{\operator{A}}

\newcommand{\F}{\operator{F}}   

\newcommand{\R}{\operator{R}}
\newcommand{\J}{\operator{J}}
\renewcommand{\H}{\operator{H}}
\newcommand{\T}{\operator{T}}
\newcommand{\G}{\operator{G}}

\renewcommand{\d}{\operator{d}}

\newcommand{\dt}{\,\d t}
\newcommand{\dx}{\,\d\bm x}
\newcommand{\jmp}[1]{\left\llbracket#1\right\rrbracket}

\newcommand{\dprod}[1]{\left<#1\right>}

\newcommand{\rec}[1]{\widetilde{u}^{#1}}

\newcommand{\un}[1]{{u}_N^{#1}}
\renewcommand{\u}{\mathfrak{u}}

\newcommand{\muf}[1]{\mu\left(\left|\nabla #1 \right|^2\right)}
\newcommand{\twon}[2]{\norm{#1}_{L^2(#2)}}
\newcommand{\diam}[1]{\mathrm{diam}({#1})}

\newcommand{\abd}{\beta}
\newcommand{\aco}{\alpha}

\newcommand{\Flc}{L_\operator{F}}
\newcommand{\LFp}{L_{\F'}}

\newcommand{\dpa}{\delta}
\newcommand{\Fsm}{\nu}

\newcommand{\coc}{C_\mathsf{I}}

\newcommand{\fpbd}{\beta_{\F'}}
\newcommand{\fpco}{\alpha_{\F'}}
\newcommand{\tol}{\epsilon_{\textrm{tol}}}

\newtheorem{theorem}{Theorem}[section]
\newtheorem{lemma}[theorem]{Lemma}
\newtheorem{proposition}[theorem]{Proposition}

\theoremstyle{definition}

\newtheorem{remark}[theorem]{Remark}

\title[Adaptive ILG Methods for Nonlinear Problems]{Adaptive Iterative Linearization Galerkin Methods\\ for Nonlinear Problems}

\author[P.~Heid \and T.~P.~Wihler]{Pascal Heid \and Thomas P.~Wihler}
\address{Mathematics Institute, University of Bern, CH-3012 Switzerland}
\email{pascal.heid@math.unibe.ch \and wihler@math.unibe.ch}

\thanks{The authors acknowledge the financial support of the Swiss National Science Foundation under grant no. 200021\underline{\phantom{x}}182524}

\begin{document}

\begin{abstract}
A wide variety of (fixed-point) iterative methods for the solution of nonlinear equations (in Hilbert spaces) exists. In many cases, such schemes can be interpreted as iterative \emph{local linearization} methods, which, as will be shown, can be obtained by applying a suitable preconditioning operator to the original (nonlinear) equation. Based on this observation, we will derive a unified abstract framework which recovers some prominent iterative schemes. In particular, for Lipschitz continuous and strongly monotone operators, we derive a general convergence analysis. Furthermore, in the context of numerical solution schemes for nonlinear partial differential equations, we propose a combination of the iterative linearization approach and the classical Galerkin discretization method, thereby giving rise to the so-called \emph{iterative linearization Galerkin (ILG)} methodology. Moreover, still on an abstract level, based on two different elliptic reconstruction techniques, we derive a posteriori error estimates which separately take into account the discretization and linearization errors. Furthermore, we propose an adaptive algorithm, which provides an efficient interplay between these two effects. In addition, the ILG approach will be applied to the specific context of finite element discretizations of quasilinear elliptic equations, and some numerical experiments will be performed.
\end{abstract}

\keywords{%
Numerical solution methods for nonlinear PDE,
monotone problems,
fixed point iterations,
linearization schemes,
Ka\v{c}anov method,
Newton method,
Galerkin discretizations,
adaptive finite element methods,
a posteriori error estimation%
}

\subjclass[2010]{35J62, 47J25, 47H05, 47H10, 49M15, 65J15, 65N30}

\maketitle

\section{Introduction}

The aim of this paper is to establish a general (adaptive) \emph{iterative linearization Galerkin (ILG)} framework for the numerical solution of nonlinear problems, with application to second-order partial differential equations (PDE) in divergence form. To set the stage, we consider a real Hilbert space $X$ with inner product~$(\cdot,\cdot)_X$ and induced norm denoted by~$\|\cdot\|_X$. We remark that for most of our work it is sufficient for $X$ to be a reflexive Banach space. Then, given a nonlinear operator~$\F:\,X\to X'$, we focus on the equation
\begin{equation}\label{eq:F=0}
u\in X:\qquad\F(u)=0 \quad\text{in }X',
\end{equation} 
where~$X'$ denotes the dual space of~$X$. In weak form, this problem reads
\begin{equation}\label{eq:F=0weak}
u\in X:\qquad \dprod{\F(u),v}=0\quad\text{for all } v\in X,
\end{equation}
with~$\dprod{\cdot,\cdot}$ signifying the duality pairing in~$X'\times X$. 

\subsubsection*{Iterative linearization}
The development of an \emph{iterative linearization scheme} for~\eqref{eq:F=0} is based on applying suitable preconditioning operators. More precisely, for given~$v\in X$, we introduce a linear and invertible \emph{preconditioning operator} \begin{equation}\label{eq:prec}
\A(v):\,X\to X',
\end{equation}
which allows to transform~\eqref{eq:F=0} into $\A(u)^{-1}\F(u)=0$.
This in turn gives rise to a fixed point iteration
\[
u^{n+1}=u^n-\A(u^n)^{-1}\F(u^n),\qquad n\ge 0,
\]
for an initial guess~$u^0\in X$, or equivalently,
\begin{equation}\label{eq:fp0}
u^{n+1}\in X:\qquad \A(u^n)u^{n+1}=\A(u^n)u^n-\F(u^n),\qquad n\ge 0.
\end{equation}
Letting
\begin{equation}\label{eq:f(u)}
f:X \to X',\qquad f(u):=\A(u)u-\F(u),
\end{equation}
the fixed-point iteration~\eqref{eq:fp0} takes the form of the following \emph{iterative linearization scheme}:
\begin{align}\label{eq:it}
\A(u^n)u^{n+1}=f(u^n), \qquad n \geq 0.
\end{align}
We emphasize that, given $u^n\in X$, this is a \emph{linear} problem for~$u^{n+1}\in X$.

The general iteration scheme~\eqref{eq:it} recovers some of the widely used fixed-point iterations occurring in the literature. These include, for instance, the Zarantonello iteration, the Ka\v{c}anov scheme, and the Newton method; see~Section~\ref{sec:particulariterations} for a detailed discussion. In the context of the Zarantonello iteration, the interested reader is referred to the original work~\cite{Zarantonello:60} (cf.~also~\cite{Browder:65} for a generalization), or the monographs~\cite[\S3.3]{Necas:86} and~\cite[\S25.4]{Zeidler:90}. Incidentally, the latter two references also deal with the Ka\v{c}anov approach, see \cite[\S4.5]{Necas:86} or~\cite[\S25.14]{Zeidler:90}. For the (damped and adaptive) Newton method we refer to~\cite{Deuflhard:04} for an extensive overview, or the recent works on adaptive Newton schemes~\cite{AmreinWihler:14,AmreinWihler:15,HoustonWihler:18,Potschka:16,SchneebeliWihler:11}.

\subsubsection*{Iterative linearized Galerkin approach} 
The iteration~\eqref{eq:it} generates a sequence $\{u^n\}_{n \geq 0}$ which potentially converges to a solution~$u^\star\in X$ of \eqref{eq:F=0}. In general, however, the computation of this sequence is not feasible if~$X$ is infinite- or high-dimensional. Therefore, in order to cast the iterative linearization approach described above into a computational framework, we will consider Galerkin discretizations of~\eqref{eq:it} in terms of finite-dimensional conforming subspaces~$X_N\subset X$. Then, a discrete approximation, $u_N^{n+1}\in X_N$, based on a starting guess~$u_N^0\in X_N$, is obtained by solving the \emph{linear} discrete system
\begin{equation}\label{eq:itN1}
u^{n+1}_N\in X_N:\qquad\dprod{\A(u_N^n)u_N^{n+1},v}=\dprod{f(u_N^n),v}\quad\forall v\in X_N,\qquad n\ge 0.
\end{equation}
We note that the discretization of the linearized problem~\eqref{eq:it} coincides with the linearization of the discretized problem~\eqref{eq:itNnonl}, i.e. the discretization and linearization commute; see~\cite{El-AlaouiErnVohralik:11} for a related discussion. For the resulting sequence~$\{u_N^n\}_{n\ge 0}\subset X_N$ of discrete solutions it is possible, under certain conditions, to obtain general a posteriori estimates for the difference to the exact solution, $u^\star\in X$, i.e. for~$\|u^\star-u_N^{n+1}\|_X$, $n\ge 0$. The emphasis of such bounds is that they enable the individual identification of different sources of error in the approximation process, such as, e.g., the linearization and discretization errors (further errors, not to be considered here, may result, for instance, from a linear solver iteration, see, e.g.,~\cite{ErnVohralik:13}, or from quadrature). This can be accomplished by means of two conceptionally different techniques, both of which will be presented in this work:
\begin{enumerate}[(a)]
\item The first approach is based on the assumption that a computable bound for the residual of the \emph{linear} Galerkin discretization of the form~\eqref{eq:itN1} is available. Then, applying an elliptic reconstruction technique (see, e.g., \cite{LakkisMakridakis:06,MakridakisNochetto:03}) yields a computable a posteriori error estimate for the error $\|u^\star-u_N^{n+1}\|_X$, which can be expressed in terms of a discretization and linearization contribution. In fact, these estimators can also be applied to appropriately enrich the space $X_N$, thereby leading to a new space $X_{N+1}$. We note that this approach has been applied previously in~\cite{CongreveWihler:17} in the specific context of the Zarantonello iteration scheme.
\item Alternatively, we may consider, for $n \geq 0$, a \emph{nonlinear} discrete problem which, on the one hand, features the \emph{nonlinear} operator~$\F$ from~\eqref{eq:F=0}, and, on the other hand, possesses the same solution, $\un{n+1}\in X_N$, as the \emph{linear} Galerkin formulation~\eqref{eq:itN1}. Assuming that there exists a computable bound for the residual of the discrete solution to a suitably reconstructed nonlinear problem, our analysis will show that such a bound can be exploited for the purpose of deriving an a posteriori error estimator.
\end{enumerate}

A posteriori error estimates as outlined above constitute an essential building block in the development of adaptive ILG schemes for nonlinear problems~\eqref{eq:F=0}. Indeed, recalling that such bounds allow to distinguish the different sources of error in the approximation process, the key idea of the fully adaptive ILG methodology is to provide an appropriate \emph{interplay} between the fixed-point linearization iteration and possible Galerkin space enrichments (e.g., mesh refinements for finite elements) depending on whether the discretization error or the linearization error is dominant. In this way, the goal of the adaptive ILG approach is to keep the number of fixed-point iterations at a minimum in the sense that no unnecessary iterations are performed if they are not expected to contribute a substantial reduction of the error on the actual Galerkin space. 

The simultaneous control of different sources of error in the context of adaptive finite element methods for monotone problems has been presented in a number of earlier papers. For instance, in the work~\cite{ChaillouSuri:07}, the authors have considered general linearizations of strongly monotone operators, and have derived computable a posteriori estimators for the total error (consisting of the linearization error and the Galerkin error) with identifiable components for each of the error sources. For even more sophisticated a posteriori error estimators in the specific context of the Newton linearization scheme for second-order monotone quasilinear diffusion problems we refer to~\cite{El-AlaouiErnVohralik:11,ErnVohralik:13}. The a posteriori error analysis derived in those papers includes---in addition to the discretization and linearization errors---also the algebraic linear solver error; moreover, the authors have proposed an adaptive iterative procedure, which takes into account all components of the numerical scheme in each refinement step. For a further development of that research in the context of compositional two-phase flow with nonlinear complementarity constraints we refer to \cite{MartinVohralik:19}. Furthermore, first a posteriori error estimates in the framework of the Ka\v{c}anov iteration for quasilinear diffusion problems in divergence form have been presented in~\cite{HaJeSh:97}. Later on, an adaptive iterative linearized Galerkin type approach has been introduced and discussed in~\cite{GarauMorinZuppa:11}; indeed, the convergence of the Ka\v{c}anov-Galerkin iteration is proved therein. Moreover, for semilinear second-order elliptic problems, two different linearization schemes of Ka\v{c}anov type have been analyzed in~\cite{BernardiDakroubMansourSayah:15}. Just recently, based on the ILG approach in~\cite{CongreveWihler:17}, the convergence of an adaptive Zarantonello-Galerkin iterative scheme for monotone elliptic PDE has been proved in~\cite{GantnerHaberlPraetoriusStiftner:17}. Finally, we point to the fact that the ILG methodology has been applied also to high-order (so-called $hp$)~\cite{AmreinMelenkWihler:17} and discontinuous Galerkin~\cite{HoustonWihler:18} finite element discretizations, as well as to nonlinear parabolic problems~\cite{AmreinWihler:17}.

\subsubsection*{Outline of the paper} 
In Section~\ref{sec:itlin} we state and prove a global convergence result for the unified iteration scheme~\eqref{eq:it}. In particular,  in order to provide a few examples, we apply our result to the Zarantonello, Ka\v{c}anov, and (damped) Newton methods, thereby recovering some of the well-known convergence results from the literature. Furthermore, still on an abstract level, in Section~\ref{sec:aposteriorianal} we discuss conforming Galerkin discretizations of~\eqref{eq:it}, and present general a posteriori error estimates based on the two approaches outlined in~(a) and~(b) above. On that account, we propose in Section~\ref{sec:adaptivealgorithm} a fully adaptive algorithm based on the a posteriori error estimates. More specifically, in Section~\ref{sec:scl}, we derive computable error bounds for a second-order PDE in divergence form; finally, in Section~\ref{sec:numexp}, these theoretical estimates are employed within a series of numerical experiments in the framework of the fully adaptive ILG approach.

\section{Iterative linearization} \label{sec:itlin}

The goal of this section is to prove a general convergence result for the iterative linearization iteration~\eqref{eq:it} under the condition that~$\F$ in~\eqref{eq:F=0} is a Lipschitz continuous and strongly monotone operator. Furthermore, we will review a few classical examples.

\subsection{Abstract framework}\label{sec:abstract}
For the purpose of this work, we restrict ourselves to Lipschitz continuous, strongly monotone operators~$\F$:
\begin{enumerate}[(F1)]
\item The operator $\F$ is Lipschitz continuous, i.e.~there exists a constant $\Flc>0$ such that 
\begin{align*}
\left|\dprod{\F(u)-\F(v),w}\right| \leq \Flc \norm{u-v}_X \norm{w}_X,
\end{align*}
for all $u,v,w \in X$.
\item The operator $\F$ is strongly monotone, i.e. there exists a constant $\Fsm>0$ such that 
\begin{align*}
 \Fsm \norm{u-v}_X^2 \leq \dprod{\F(u)-\F(v),u-v},
\end{align*}
for all $u,v \in X$.
\end{enumerate}
Under these conditions, the theory of strongly monotone operators implies that~\eqref{eq:F=0} possesses a unique solution~$u^\star\in X$; see, e.g., \cite[\S3.3]{Necas:86} or~\cite[\S25.4]{Zeidler:90}.

Furthermore, for given~$u\in X$, we introduce the bilinear form
\begin{equation}\label{eq:Aweak}
a(u;v,w):=\dprod{\A(u)v,w}, \qquad v,w\in X,
\end{equation}
where~$\A(\cdot)$ is the preconditioning operator from~\eqref{eq:prec}. Then,
we can write~\eqref{eq:it} in weak form: given~$u^n\in X$,  find~$u^{n+1}\in X$ such that
\begin{equation}\label{eq:itweak}
a(u^n;u^{n+1},w)=\dprod{f(u^n),w}\qquad \forall w\in X.
\end{equation}
Throughout this paper, for any~$u\in X$, we assume that the bilinear form~$a(u;\cdot,\cdot)$ is uniformly coercive and bounded. Those assumptions refer to the fact that there are two constants~$\alpha,\beta>0$ independent of $u \in X$, such that
\begin{equation}\label{eq:coercive}
a(u;v,v) \geq \aco \|v\|_X^2 \qquad\forall v \in X,
\end{equation}
and
\begin{equation}\label{eq:continuity}
a(u;v,w) \leq \abd \norm{v}_X \norm{w}_X \qquad\forall  v,w \in X,
\end{equation}
respectively. In particular, owing to the Lax-Milgram Theorem, these properties imply the well-posedness of the solution~$u^{n+1}\in X$ of the linear equation~\eqref{eq:it}, for any given~$u^n\in X$. 

\subsection{A global convergence result} 

Given the framework introduced in the previous Section~\ref{sec:abstract}, the ensuing proposition is an abstract global convergence result for the iteration scheme~\eqref{eq:it}. We note that it can be extended readily to the case where $X$ is a reflexive Banach space.

\begin{proposition} \label{pro:kacanovgeneral}
Suppose that~{\rm (F2)} (cf.~Section~\ref{sec:abstract}), \eqref{eq:coercive} and~\eqref{eq:continuity} are satisfied, and $u \mapsto a(u;u,\cdot)$ and $u \mapsto \F(u)$ are continuous mappings from $X$ into its dual space~$X'$ with respect to the weak topology on $X'$. If the sequence $\{u^{n}\}_{n \geq 0}$ defined by \eqref{eq:it} satisfies~$\|u^{n+1}-u^n\|_X\to0$ as~$n\to\infty$, then it converges to the unique solution $u^\star\in X$ of~\eqref{eq:F=0}.
\end{proposition}

\begin{proof}
We begin by showing that $\{u^n\}_{n \geq 0}$ is a Cauchy sequence. Indeed, by virtue of~(F2) and~\eqref{eq:f(u)}, for any $m \geq n \ge0$, it holds that
\begin{align*}
\nu \norm{u^{m}-u^n}_X^2
& \leq \dprod{\F(u^m)-\F(u^n),u^m-u^n}\\
&=\dprod{\A(u^m)u^m-f(u^m),u^m-u^n}-\dprod{\A(u^n)u^n-f(u^n),u^m-u^n}.
\end{align*}
Hence, involving~\eqref{eq:Aweak} and~\eqref{eq:itweak} gives
\begin{align*}
\nu \norm{u^{m}-u^n}_X^2
&\le a(u^{m};u^{m}-u^{m+1},u^m-u^n)-a(u^{n};u^{n}-u^{n+1},u^m-u^n).
\end{align*}
Furthermore, \eqref{eq:continuity} implies that 
\begin{align*}
\norm{u^{m}-u^n}_X \leq \frac{\beta}{\nu} \left(\norm{u^{m+1}-u^{m}}_X + \norm{u^{n+1}-u^n}_X\right)\to 0,
\end{align*}
for $n,m \to \infty$. Hence, $\{u^n\}_{n \geq 0}$ is a Cauchy sequence, and, therefore, converges to some limit $u^\star \in X$.  Next, we show that $u^\star$ is the unique solution of \eqref{eq:F=0}. Owing to~ \eqref{eq:itweak}, we notice the identity
\begin{align*}
 a(u^n;u^n,v)-\dprod{f(u^n),v}+a(u^n;u^{n+1}-u^n,v)=0\qquad\forall v\in X,
\end{align*}
for all $n \geq 0$. Here, due to~\eqref{eq:continuity}, and because $\norm{u^{n+1}-u^n}_X$ is a vanishing sequence, we observe that
\[
a(u^n;u^n,v)-\dprod{f(u^n),v}\to 0\qquad\forall v\in X,
\]
for~$n\to\infty$. Hence, by continuity of~$a$ and~$f$, we deduce that
\begin{align*}
 a(u^\star;u^\star,v)=\dprod{f(u^\star),v} \qquad \forall v \in X,
\end{align*}
i.e.~$u^\star$ is a solution of \eqref{eq:F=0}; we note that $u \mapsto f(u)=a(u;u,\cdot)-\F(u)$ is continuous by the continuity of $a$ and $\F$. It remains to show that~$u^\star$ is the only solution of~\eqref{eq:F=0}. In fact, if $u^\square\in X$ is any other solution, then~(F2) leads to 
\begin{align*}
 \nu \norm{u^\star-u^\square}_X^2 &\leq \dprod{\F(u^\star)-\F(u^\square),u^\star-u^\square}=0,
\end{align*}
i.e.~$u^\star=u^\square$. 
\end{proof}

\subsection{Applications} \label{sec:particulariterations}

In the ensuing section we will discuss the general Proposition~\ref{pro:kacanovgeneral} in the context of the Zarantonello, Ka\v{c}anov, and Newton iterations.

\subsubsection{Zarantonello iteration}
A most simple choice for the preconditioning operator from~\eqref{eq:prec} is~$\A(v) u:=(\delta^{-1}u,\cdot)_X$, where~$\delta>0$ is a fixed constant; in particular, here, $\A=\A(v)$ is independent of~$v$. In this case, the iterative linearization scheme~\eqref{eq:it} turns out to be
\begin{align}\label{eq:zarantonelloit}
 (u^{n+1},\cdot)_X=(u^{n},\cdot)_X-\dpa \dprod{\F(u^{n}),\cdot}.
\end{align}

\begin{theorem}[Convergence of the Zarantonello iteration]\label{thm:zarantonello}
Assuming~{\rm (F1)} and {\rm (F2)} (cf.~Section~\ref{sec:abstract}), the Zarantonello iteration \eqref{eq:zarantonelloit} converges to the unique solution $u^\star$ of~\eqref{eq:F=0} for any $\dpa \in \left]0,\nicefrac{2 \Fsm}{\Flc^2}\right[$.
\end{theorem}

\begin{proof}
We verify the assumptions required for Proposition~\ref{pro:kacanovgeneral} to hold. For~$a(u,v)=(\delta^{-1}u,v)_X$, $u,v\in X$, we note that~\eqref{eq:coercive} and~\eqref{eq:continuity} are satisfied with
\begin{equation}\label{eq:Zab}
\alpha=\beta=\delta^{-1}>0. 
\end{equation}
Moreover, both $u \mapsto a(u,\cdot)=(\dpa^{-1}u,\cdot)_X$ and $u \mapsto \F(u)$ are continuous on~$X$. It remains to show that $\norm{u^{n+1}-u^n}_X$ vanishes. For that purpose, we denote by $\J:X \to X'$ the Riesz-Fr\'{e}chet isometry. The iteration \eqref{eq:zarantonelloit} can then be written, in strong form, as $u^{n+1}=\T(u^n)$, where $\T(u):=u-\dpa \J^{-1}\F(u)$. This leads to
\begin{align*}
 \norm{u^{n+1}-u^n}_X^2 &= \norm{\T(u^n)-\T(u^{n-1})}_X^2 \\
 &= \norm{u^n-u^{n-1}}_X^2 -2 \dpa \dprod{\F(u^{n})-\F(u^{n-1}),u^n-u^{n-1}}\\
 & \quad + \dpa^2 \norm{\J^{-1}(\F(u^n)-\F(u^{n-1}))}_X^2,
\end{align*}
where we have used the linearity of $\J^{-1}$. Invoking (F1) and (F2), together with the fact that $\J^{-1}$ is isometric, we further get
\begin{align*}
 \norm{u^{n+1}-u^n}_X^2 \leq \left(1-2 \dpa \Fsm+\dpa^2 \Flc^2\right)\norm{u^{n}-u^{n-1}}_X^2. 
\end{align*}
We note that 
\begin{equation}\label{eq:gamma1}
\gamma:=\left(1-2 \dpa \Fsm+\dpa^2 \Flc^2\right)<1
\end{equation} 
if and only if $\dpa \in \left]0,\nicefrac{2 \Fsm}{\Flc^2}\right[$. Hence, by induction,
\[
\norm{u^{n+1}-u^n}_X^2\le \gamma^n\norm{u^{1}-u^0}_X^2,
\]
which shows that~$\|u^{n+1}-u^n\|_X\to 0$ as~$n\to 0$.
\end{proof}

\begin{remark}
We notice that the contraction factor~$\gamma$ from~\eqref{eq:gamma1} is minimal for the choice~$\delta=\nicefrac{\nu}{\Flc^2}$.
\end{remark}

\subsubsection{Ka\v{c}anov iteration}

Here we assume that the nonlinear operator~$\F$ from~\eqref{eq:F=0} takes the form $\F(u)=\A(u)u-g$, where $\A(u):\,X\to X'$ is linear (for given $u\in X$), and $g=-\F(0) \in X'$ is fixed. Then, the Ka\v{c}anov iteration is defined by
\begin{align} \label{eq:kacanovstrong}
\A(u^n)u^{n+1}=g,\qquad n\ge 0.
\end{align}
Note that this iteration can be cast into the setting of~\eqref{eq:it}, where~$\A(u^n)$ takes the role of the preconditioning operator, and $f(u^n)=\A(u^n)u^n-\F(u^n)=g$ is constant.
We make the assumption that
there exists a Gateaux differentiable functional $\G:X \to \mathbb{R}$ which satisfies the following properties:
\begin{enumerate}[(K1)]
\item $\G'(u)=a(u;u,\cdot)$ on~$X$, and $\G'$ is continuous and strongly monotone, i.e. there exists a real number $c_0 >0$ such that, for any~$u,v\in X$, it holds
\begin{align}\label{eq:Gmono}
\dprod{\G'(u)-\G'(v),u-v} \geq c_0 \norm{u-v}_X^2;
\end{align}
\item  For each $u,v\in X$ we have the bound $\G(u)-\G(v) \geq \nicefrac12 \left(a(u;u,u)-a(u;v,v)\right)$.
\end{enumerate}

In order to be able to apply Proposition~\ref{pro:kacanovgeneral}, we need an auxiliary result, which will also be crucial in the analysis of the Newton method in Section~\ref{sc:Newton} below.

\begin{lemma} \label{lemma:Hdifference}
If $\H:X \to \mathbb{R}$ is Gateaux differentiable with $\H'$ continuous and strongly monotone, then $\H$ is bounded from below.
\end{lemma}

\begin{proof}
For fixed $v \in X$, and~$t\in[0,1]$, we define the function $\varphi(t):=\H(tv)$. We note that $\varphi'(t)=\dprod{\H'(tv),v}$, and, invoking the fundamental theorem of calculus, we find that
\begin{equation}\label{eq:H1}
 \H(v)-\H(0)=\int_0^1 \dprod{\H'(tv),v} \dt = \int_0^1 \dprod{\H'(tv)-\H'(0),v} \dt + \dprod{\H'(0),v}.
\end{equation}
Since $\H'$ is strongly monotone, there exists a constant~$\gamma>0$ such that
\begin{align*}
 \dprod{\H'(tv)-\H'(0),v}&=\frac{1}{t} \dprod{\H'(tv)-\H'(0),tv} \geq \gamma t \norm{v}_X^2,
\end{align*}
for any~$t \in ]0,1]$. Inserting this bound into~\eqref{eq:H1}, integrating with respect to~$t$, and using the submultiplicativity of the operator norm, yields
\begin{align*}
 \H(v) \geq \frac{\gamma}{2} \norm{v}_X^2-\norm{\H'(0)}_{X'} \norm{v}_X + \H(0).
\end{align*}
It is elementary to verify that the right-hand side is minimal for $\norm{v}_X=\gamma^{-1}\norm{\H'(0)}_{X'}$. With this choice we arrive at $\H(v) \geq \H(0)-\nicefrac{1}{2\gamma}\norm{\H'(0)}^2_{X'}$ for all~$v \in X$,
i.e.~$\H$ is bounded from below.
\end{proof}

\begin{theorem}[Convergence of the Ka\v{c}anov iteration] \label{thm:kacanov}
Suppose that {\rm (K1)} and~{\rm (K2)} hold. Furthermore, assume that the bilinear form $a(u;\cdot,\cdot)$ induced by $\A$ satisfies \eqref{eq:coercive} and \eqref{eq:continuity}, and is symmetric for all $u \in X$. Then the sequence $\{u^{n}\}_{n \geq 0}$ defined by \eqref{eq:kacanovstrong} converges to the unique solution $u^\star$ of \eqref{eq:F=0}.
\end{theorem}

\begin{proof}
Because of~(K1) it follows that~$u\mapsto a(u;u,\cdot)=\operator{G}'(u)$ is continuous. Moreover, $u\mapsto f(u)$ is constant, and thus continuous. Consequently, $u \mapsto \F(u)=a(u;u,\cdot)-f(u)$ is continuous as well.
We show that $\norm{u^{n+1}-u^n}_X$, $n\ge 0$, is a vanishing sequence. To this end, we follow closely along the lines of the proof of~\cite[Theorem~25.L]{Zeidler:90}. Let us introduce the functional $\H(u):=\G(u)-\dprod{g,u}$. We note that 
$\H'(u)=\G'(u)-g=\A(u)u-g=\F(u)$, i.e.~$\H$ is the potential of $\F$. Moreover, by virtue of~(K1), the derivative $\H'=\G'-g$ is continuous and strongly monotone, and thus $\F$ satisfies~(F2). In particular, with the aid of Lemma~\ref{lemma:Hdifference}, we deduce that $\H$ is bounded from below. Next, we will verify that $\left\{\H(u^n)\right\}_{n \geq 0}$ is a monotone decreasing sequence. Indeed, noticing that $a(u^n;u^{n+1},u^{n+1}-u^n)=\dprod{g,u^{n+1}-u^n}$, and employing~(K2), yields
\begin{align*}
 \H(u^n)-\H(u^{n+1})
 &= \dprod{g,u^{n+1}-u^n}+\G(u^{n})-\G(u^{n+1})\\
 &\geq a(u^n;u^{n+1},u^{n+1}-u^n)+\frac12a(u^n;u^n,u^n)-\frac12a(u^n;u^{n+1},u^{n+1})\\
 &\geq \frac12a(u^n;u^n,u^n)-a(u^n;u^{n+1},u^n)+\frac12a(u^n;u^{n+1},u^{n+1}),
 \end{align*}
for any~$n\ge 0$. Then, employing the symmetry of~$a(u^n;\cdot,\cdot)$, and involving~\eqref{eq:coercive}, we obtain
\begin{equation}\label{eq:H2}
 \H(u^n)-\H(u^{n+1})
 \ge \frac12a(u^n;u^{n+1}-u^n,u^{n+1}-u^n)
 \ge\frac{\alpha}{2}\norm{u^{n+1}-u^n}_X^2\ge 0,
\end{equation}
which shows that~$\{\H(u^n)\}_{n\ge 0}$ is monotone decreasing. Then, recalling the  boundedness from below, we conclude that $\H(u^n)-\H(u^{n+1})\to0$ as~$n\to\infty$. Hence, exploiting~\eqref{eq:H2}, it follows that~$\norm{u^{n+1}-u^n}_X$ vanishes, and the proof is complete.
\end{proof}

\subsubsection{Newton iteration}\label{sc:Newton}

For the Newton iteration the preconditioning operator in~\eqref{eq:fp0} is selected to be~$\A(v)=\delta(v)^{-1}\F'(v)$, $v\in X$, where~$\delta(v)>0$ is a (damping) parameter, and $\F'$ signifies the Gateaux derivative of~$\F$. Then, the (damped) Newton iteration is given by
\begin{align} \label{eq:newtonstrong}
 \F'(u^{n})u^{n+1}=\F'(u^{n})u^{n}-\dpa(u^n) \F(u^{n}),\qquad n\ge 0.
\end{align}
For the purpose of applying Proposition~\ref{pro:kacanovgeneral}, we make the following assumptions:

\begin{enumerate}[(N1)]
\item The operator $\F$ is Gateaux differentiable.
Moreover, $\F'$ is coercive and bounded in the sense that, for any given $u \in X$, it holds
\begin{equation}\label{eq:N11}
 \dprod{\F'(u)v,v} \geq \fpco \norm{v}_X^2 \qquad \forall v \in X,
\end{equation}
and
\begin{equation}\label{eq:N12}
 \dprod{\F'(u)v,w} \leq \fpbd \norm{v}_X \norm{w}_X \qquad \forall v,w \in X, 
\end{equation}
where $\fpco,\fpbd >0$ are independent of $u$.
\item It exists a Gateaux differentiable functional $\operator{G}:\,X \to \mathbb{R}$ such that $\G'(u)=\F'(u)u$ in~$X'$ for any~$u\in X$, and $\G'$ is continuous when $X'$ is endowed with the weak topology.
\item It exists a Gateaux differentiable functional $\operator{H}:\,X \to \mathbb{R}$ such that $\operator{H}'=\F$. 
\item There are some constants~$0<\delta_{\min}\le\delta_{\max}<\infty$ such that $\delta:\,X\to[\delta_{\min},\delta_{\max}]$ is a continuous functional.
\end{enumerate} 

\begin{theorem}[Convergence of the damped Newton iteration] \label{thm:newton}
Assume {\rm (F1)} and~{\rm (F2)} (cf.~Section~\ref{sec:abstract}), as well as {\rm (N1)}--{\rm (N4)}. Then, for $\delta_{\max} <\nicefrac{2\fpco}{\Flc}$ in~{\rm (N4)} the damped Newton iteration \eqref{eq:newtonstrong} converges to the unique solution $u^\star\in X$ of~\eqref{eq:F=0} .
\end{theorem}

\begin{proof}
We aim at employing Proposition~\ref{pro:kacanovgeneral} as before. By virtue of~\eqref{eq:N11}, \eqref{eq:N12}, and~(N4), we obtain
\[
a(u;v,v)\ge\fpco\delta_{\max}^{-1}\norm{v}_X^2,\qquad u,v\in X,
\]
and
\[
a(u;v,w)\le\fpbd\delta_{\min}^{-1}\norm{v}_X\norm{w}_X,\qquad u,v,w\in X,
\]
which are the coercivity and boundedness conditions~\eqref{eq:coercive} and~\eqref{eq:continuity}, with
\begin{equation}\label{eq:Nab}
\alpha=\nicefrac{\fpco}{\delta_{\max}},\qquad
\beta=\nicefrac{\fpbd}{\delta_{\min}},
\end{equation}
respectively. Next, we remark that the maps~$u\mapsto a(u;u,\cdot)=\delta(u)^{-1}\F'(u)u$ and~$u\mapsto \F(u)$ are both continuous, when $X'$ is endowed with the weak topology, by (N2) and (N4), and by (F1), respectively. Therefore, by the same arguments as in the proof of Theorem~\ref{thm:kacanov}, it suffices to show that there exists a constant $C>0$ such that 
 \begin{align} \label{eq:Hdifference}
  \H(u^n)-\H(u^{n+1}) \geq C \norm{u^{n+1}-u^n}_X^2, \qquad n \geq 0.
 \end{align}
To this end, we define the function $\varphi(t):=\H(u^n+t(u^{n+1}-u^n))$, $t\in[0,1]$, and observe that 
\begin{align*}                                                                                                                                                                                        
\varphi'(t)
=\dprod{\H'(u^n+t(u^{n+1}-u^n)),u^{n+1}-u^n}
=\dprod{\F(u^n+t(u^{n+1}-u^n)),u^{n+1}-u^n}.
\end{align*}
Then, the fundamental theorem of calculus implies that
\begin{align*}
 \H(u^{n})-\H(u^{n+1})
&= -\int_0^1 \dprod{\F(u^n+t(u^{n+1}-u^n)),u^{n+1}-u^n} \dt\\
&= -\int_0^1 \dprod{\F(u^{n}+t(u^{n+1}-u^{n}))-\F(u^{n}),u^{n+1}-u^{n}} \dt \\
& \quad - \dprod{\F(u^{n}),u^{n+1}-u^{n}}.
\end{align*}
By the definition of the Newton iteration~\eqref{eq:newtonstrong}, it holds that $\F(u^{n})=\dpa(u^n)^{-1}\F'(u^{n})(u^{n}-u^{n+1})$, $n\ge 0$. Thus, with the aid of~(F1) and~\eqref{eq:N11}, it follows that 
\begin{align*}
 \H(u^{n})-\H(u^{n+1})&\geq-\Flc \int_0^1  t \norm{u^{n+1}-u^{n}}_X^2 \dt + \dpa(u^n)^{-1}\dprod{\F'(u^{n})(u^{n+1}-u^{n}),u^{n+1}-u^{n}} \\
 & \geq -\frac{\Flc}{2}\norm{u^{n+1}-u^{n}}_X^2 + \fpco\dpa(u^n)^{-1} \norm{u^{n+1}-u^{n}}_X^2.
\end{align*}
If $\delta(u^n) \leq \delta_{\max}<\nicefrac{2\fpco}{\Flc}$, then 
\begin{equation}\label{eq:C}
\frac{\fpco}{\dpa(u^n)} - \frac{\Flc}{2} \geq \frac{\fpco}{\dpa_{\max}} - \frac{\Flc}{2}=:C>0, \qquad n \geq 0.
\end{equation} 
We conclude that~\eqref{eq:Hdifference} is satisfied. 
\end{proof}

\begin{remark}[Classical Newton scheme]\label{rem:newton1}
Recalling~\eqref{eq:N11} with $u=u^0$, and applying~\cite[Theorem~3.3.23]{Necas:86}, we deduce the bound $\norm{\F'(u^0)^{-1}v}_{X}\le\fpco^{-1}\norm{v}_{X'}$ for all~$v\in X'$. Furthermore, assume that $\F$ is Fr\'{e}chet differentiable and~$\F'$ is Lipschitz continuous, i.e. there exists a constant~$\LFp>0$ such that
\[
\norm{\F'(v)-\F'(w)}_{X'}\le\LFp\norm{v-w}_X\qquad\forall v,w\in X.
\]
This leads to
\begin{align*}
 \norm{\F'(u^0)^{-1}(\F'(u)-\F'(v))}_X \leq \frac{\LFp}{\fpco}  \norm{u-v}_X\qquad\forall u,v\in X.
\end{align*}
Moreover, if the initial guess~$u^0\in X$ in~\eqref{eq:newtonstrong} is sufficiently close to the solution~$u^\star\in X$ of~\eqref{eq:F=0} in the sense that  
$\norm{\F(u^0)}_{X'} < \nicefrac{\fpco^2}{2 \LFp}$, then
we infer that
\begin{align*}
 \norm{\F'(u^0)^{-1}\F(u^0)}_X \leq \frac{1}{\fpco} \norm{\F(u^0)}_{X'} <\frac{\fpco}{2 \LFp}.
\end{align*}
Referring to \cite[Theorem 2.1]{Deuflhard:04}, it follows that the classical Newton iteration with~$\delta(u^n)=1$ in~\eqref{eq:newtonstrong} is well-defined, converges to a solution of \eqref{eq:F=0}, and converges \emph{quadratically}.
\end{remark}

\begin{remark}\label{rem:adnewton}
The proof of Theorem~\ref{thm:newton} is crucially based on~\eqref{eq:Hdifference}. We emphasize that this bound may be satisfied even if the damping parameter~$\delta(u^n)$ in~\eqref{eq:newtonstrong} is larger than $\nicefrac{2\fpco}{\Flc}$. This is particularly important when~$\nicefrac{2\fpco}{\Flc}\le 1$, and the choice~$\delta(u^n)=1$ (leading to local quadratic convergence, cf.~Remark~\ref{rem:newton1}) is not admissible \emph{a priori}. In this case, we may fix~$\epsilon>0$ small, and aim to \emph{a posteriori} attain the bound, for~$n\ge 0$,
\begin{align} \label{eq:newtoncondition}
\H(u^n)-\H(u^{n+1}) \geq \epsilon \norm{u^{n+1}-u^n}_X^2.
\end{align}  
To this end, we may pursue, for instance, the adaptive damping parameter selection approach proposed in~\cite[\S 3.1]{Deuflhard:04}. More precisely, in each iterative step, we define an initial value for $\delta(u^{n})$ by the following \emph{prediction} strategy:
\begin{align*}
\dpa^{n,0}=\begin{cases} \min\left(\nicefrac{\dpa(u^{n-1})}{\kappa},1\right) & \text{if } \dpa(u^{n-2}) \leq \dpa(u^{n-1}), \\
\dpa(u^{n-1}) & \text{else}. \end{cases}
\end{align*}
where $0<\kappa<1$ is a fixed (correction) factor. Here, we set $\dpa(u^{-2})=\dpa(u^{-1})=\delta^0$, with~$\delta^0$ an initial choice. 
If $u^{n+1}$ is obtained by the damped Newton method with damping parameter $\dpa^{n,i}$, for some~$i\ge 0$, then we need to verify wether or not \eqref{eq:newtoncondition} is satisfied. If not, then we adjust the damping parameter according to the \emph{correction} strategy
\begin{align}\label{eq:i}
\dpa^{n,i+1}=\max\left(\fpco(\epsilon+\nicefrac{\Flc}{2})^{-1}, \kappa \dpa^{n,i}\right),\qquad i\ge 0.
\end{align}
Subsequently, we will compute $u^{n+1}$ for the new choice~$\delta^{n,i+1}$. This process is repeated until~\eqref{eq:newtoncondition} is true, say after~$i^n$ iterations of~\eqref{eq:i}. At this point, we let~$\delta(u^n):=\delta^{n,i^n}$. 
Evidently, in view of~\eqref{eq:C}, we remark that~\eqref{eq:newtoncondition} will certainly hold once~$\delta^{n,i}\le\fpco(\epsilon+\nicefrac{\Flc}{2})^{-1}$. Moreover, in the Galerkin setting, we note that~\eqref{eq:N11} can be verified numerically at the cost of an eigenvalue problem. Finally, if the values of the constants $\fpco$ and $\Flc$ are not easily accessible, we can simply use the (possibly pessimistic) damping parameter $\dpa^{n,i+1}:=\kappa\dpa^{n,i}$.
\end{remark}

\begin{remark}
We emphasize that the proof of Theorem~\ref{thm:newton} works for much more general preconditioning operators~$\A$. Assume that $\F$ satisfies (F1), (F2), and (N3), the mapping $u \mapsto \A(u)u=a(u;u,\cdot)$ is continuous w.r.t.~the weak topology on $X'$, and the bilinear form induced by $\A$ fulfills~\eqref{eq:coercive} and \eqref{eq:continuity}. If $\alpha>\nicefrac{\Flc}{2}$, then the crucial property~\eqref{eq:Hdifference} holds with $C:=\alpha-\nicefrac{\Flc}{2}>0$; indeed, this can be shown as in the proof of Theorem~\ref{thm:newton}, and the convergence of the method can be proved similarly as before. In particular, our unified iteration scheme does also recover Newton-like methods, e.g., the case $\A(u):=\delta \F'(u_0)$ for some initial guess $u_0 \in X$, with a small enough damping parameter $\delta>0$.        
\end{remark}


\section{Galerkin approach and a posteriori error analysis} \label{sec:aposteriorianal}

The numerical solution of~\eqref{eq:F=0} is based on a finite-dimensional subspace~$X_N\subset X$, and on the iterative linearization Galerkin (ILG) formulation~\eqref{eq:itN1}, with a given initial guess~$u_N^0\in X_N$. Since~$X_N\subset X$, the assumptions in Section~\ref{sec:abstract} guarantee the existence of~$\un{n+1}\in X_N$ in each iteration step.

In this section, we will pursue two different strategies for the derivation of a posteriori error estimates for $\norm{u^\star-\un{n+1}}_X$, where $u^\star \in X$ is the unique solution of \eqref{eq:F=0}. In both approaches an \textit{elliptic reconstruction} technique, cf.~\cite{LakkisMakridakis:06,MakridakisNochetto:03}, will be employed. In the first method we use an elliptic reconstruction for the solution of the \emph{linear problem~\eqref{eq:itN1}}, and the second strategy is based on applying a similar idea for a nonlinear discrete problem equivalent to~\eqref{eq:itN1}. We will refer to this methods as the \emph{linear} and \emph{nonlinear elliptic reconstruction}, respectively.


\subsection{A posteriori error analysis based on a \emph{linear} elliptic reconstruction} 

For the sake of a general a posteriori error analysis, using a linear elliptic reconstruction, we suppose that there exists a computable bound $\eta(\un{n+1},\un{n})$ for the residual
\begin{align} \label{eq:linresbound}
 \sup_{\substack{v \in X \\ \norm{v}_X=1}} \left\{a(\un{n};\un{n+1},v)-\dprod{f(\un{n}),v}\right\} \leq \eta(\un{n+1},\un{n}).
\end{align}
We remark that, in the context of finite element methods for linear elliptic problems, there is a large body of literature focusing on the development of such estimates; see, e.g., \cite{AinsworthOden:96,Verfurth:13}.

\begin{theorem}\label{thm:aposteriorierror}
Suppose that {\rm  (F1)} and {\rm (F2)}, cf.~Section~\ref{sec:abstract}, as well as~\eqref{eq:coercive} and~\eqref{eq:continuity} hold true. Then, we have the a posteriori error bound
\begin{align}\label{eq:linestimator}
\norm{u^\star-\un{n+1}}_X \leq \frac{\beta}{\alpha \nu} \eta(\un{n+1},\un{n}) + \frac{\beta+\Flc}{\nu} \norm{\un{n+1}-\un{n}}_X,
\end{align}
where $u^\star$ is the unique solution of \eqref{eq:F=0}.
\end{theorem}

\begin{proof}
Due to \eqref{eq:coercive} and~\eqref{eq:continuity} there exists a unique $\widetilde{u}^{n+1} \in X$ such that
\begin{align}\label{eq:reconit}
a(u_N^n;\rec{n+1},v)=\dprod{f(\un{n}),v} \qquad \forall v \in X.
\end{align} 
We note that $\rec{n+1}$ is a reconstruction in the sense that $u_N^{n+1}\in X_N$ is the Galerkin projection of $\rec{n+1}$. By using the assumption (F2), we find that 
\begin{align*}
\nu \norm{u^\star-\un{n+1}}_X^2 &\leq \dprod{\F(u^\star)-\F(\un{n+1}),u^\star-\un{n+1}}
=- \dprod{\F(\un{n+1}),u^\star-\un{n+1}},
\end{align*}
since~$u^\star\in X$ is the solution of~\eqref{eq:F=0}. Hence, 
\begin{align*}
\nu \norm{u^\star-\un{n+1}}_X^2 &\leq -a(\un{n};\un{n+1},u^\star-\un{n+1})+\dprod{f(\un{n}),u^\star-\un{n+1}} \\
& \quad + a(\un{n};\un{n+1}-\un{n},u^\star-\un{n+1})\\
& \quad + a(\un{n};\un{n},u^\star-\un{n+1})- \dprod{f(\un{n}),u^\star-\un{n+1}} \\
& \quad - \dprod{\F(\un{n+1}),u^\star-\un{n+1}}.
\end{align*}
Using~\eqref{eq:reconit} and~\eqref{eq:f(u)}, this estimate transforms into
\begin{align*}
\nu \norm{u^\star-\un{n+1}}_X^2 &\leq a(\un{n};\rec{n+1}-\un{n+1},u^\star-\un{n+1})
 + a(\un{n};\un{n+1}-\un{n},u^\star-\un{n+1}) \\
& \quad - \dprod{\F(\un{n+1})-\F(\un{n}),u^\star-\un{n+1}}.
\end{align*}
Applying~\eqref{eq:continuity} and~(F1), we find that 
\begin{align*}
\nu \norm{u^\star-\un{n+1}}_X^2 & \leq \beta \norm{\rec{n+1}-\un{n+1}}_X \norm{u^\star-\un{n+1}}_X + \beta \norm{\un{n+1}-\un{n}}_X \norm{u^\star-\un{n+1}}_X\\
&  \quad + \Flc \norm{\un{n+1}-\un{n}}_X \norm{u^\star-\un{n+1}}_X.
\end{align*}
Dividing by~$\norm{u^\star-\un{n+1}}_X$ yields
\begin{align}\label{eq:boundthmlinear}
\nu\norm{u^\star-\un{n+1}}_X \leq \beta\norm{\rec{n+1}-\un{n+1}}_X + (\beta+\Flc)\norm{\un{n+1}-\un{n}}_X.
\end{align}
Moreover, by the coercivity property~\eqref{eq:coercive}, for~$\un{n+1}\neq\rec{n+1}$, we note that
\begin{align*}
\alpha \norm{\un{n+1}-\rec{n+1}}_X
&\leq \frac{a(\un{n};\un{n+1}-\rec{n+1},\un{n+1}-\rec{n+1})}{\norm{\un{n+1}-\rec{n+1}}_X} \le\sup_{\substack{v \in X \\ \norm{v}_X=1}} a(\un{n};\un{n+1}-\rec{n+1},v).
\end{align*}
Involving \eqref{eq:reconit} and~\eqref{eq:linresbound}, we arrive at
\begin{align} \label{eq:recdiffbound}
\alpha \norm{\un{n+1}-\rec{n+1}}_X
&\le \sup_{\substack{v \in X \\ \norm{v}_X=1}}  \left(a(\un{n};\un{n+1},v)-\dprod{f(\un{n}),v}\right)
\le\eta(\un{n+1},\un{n}).
\end{align}
Inserting this estimate into \eqref{eq:boundthmlinear}, finishes the proof.
\end{proof}

\begin{remark} \label{rem:linestimator}
We emphasize that the estimator~\eqref{eq:linestimator} permits to bound the error $\norm{u^\star-\un{n+1}}_X$ separately in terms of the \emph{discretization} error indicator, $\nicefrac{\beta}{\alpha \nu} \eta(\un{n+1},\un{n})$, and of the \emph{linearization} error indicator, $\nicefrac{(\beta+\Flc)}{\nu} \norm{u_N^{n+1}-u_N^n}_X$. Let us discuss these error contributions in more detail: First, recall that $\un{n+1}$ is the Galerkin projection of $\rec{n+1}$ in the sense of the Galerkin orthogonality property
 \begin{align*}
  a(\un{n};\rec{n+1}-\un{n+1},v)=0 \qquad \forall v \in X_N.
 \end{align*}
Thus, the quantity $\norm{\rec{n+1} -\un{n+1}}_X$ is an indicator for the quality of approximation of the Galerkin discretization. Here, invoking~\eqref{eq:recdiffbound}, we see that
\[
\norm{\rec{n+1} -\un{n+1}}_X \leq \alpha^{-1} \eta(\un{n+1},\un{n}),
\]
wherefore it is reasonable to interpret the $\eta$-term as the discretization error contribution in the total estimator. Secondly, it holds that $\norm{\un{n+1}-\un{n}}_X \to 0$ for $n \to 0$, which underlines the convergence of the iterative linearization; thereby, this term can be seen to quantify the linearization effect in the ILG approximation. 
\end{remark}

\begin{remark}
The constants for the estimator in Theorem~\ref{thm:aposteriorierror} for the Zarantonello iteration can be slightly improved. This is due to the fact that the preconditioning operator~$\A$ is constant in this case; cf.~\cite[Proposition~2.2]{CongreveWihler:17}.
\end{remark}


\subsection{A posteriori error analysis based on a \emph{nonlinear} elliptic reconstruction} \label{sec:aposterioriNGM}

In this section we devise an a posteriori error estimate for the \emph{linear} Galerkin iteration~\eqref{eq:itN1} based on applying the reconstruction technique to a \emph{nonlinear} discrete problem equivalent to \eqref{eq:itN1}. We underline that the nonlinear problem~\eqref{eq:newnonlinear}, exactly as in the case of the linear elliptic reconstruction from~\eqref{eq:reconit}, is of purely theoretical relevance in the derivation of the estimator, and does not need to be solved in the actual computations. 

We define an operator $\psi_N: X \to X_N$, where, for fixed $w \in X$, we let $\psi_N(w)$ to be the Riesz representative of~$\F(w) \in X_N'$ with respect to the inner product in~$X$, i.e. 
\begin{equation}\label{eq:R1}
(\psi_N(w),v)_X=\dprod{\F(w),v} \qquad \forall v \in X_N.
\end{equation}
Note that, if $\u_N$ is the solution of the \emph{nonlinear} Galerkin approximation of~\eqref{eq:F=0weak} with respect to the discrete space~$X_N$, i.e.
\begin{equation}\label{eq:itNnonl}
\u_N\in X_N:\qquad\dprod{\F(\u_N),v}=0\qquad\forall v\in X_N,
\end{equation}
then it holds that $\psi_N(\u_N)=0$. 

For each $n \geq 0$, we define the \emph{nonlinear} elliptic reconstruction $\rec{n+1} \in X$ of the solution~$u_N^{n+1}\in X_N$ of~\eqref{eq:itN1} by
\begin{align}\label{eq:newnonlinear}
\dprod{\F(\rec{n+1}),v}=(\psi_N(\un{n+1}),v)_X \qquad \forall v \in X.
\end{align}
By construction of the operator $\psi_N$, it holds that $\un{n+1}$ is the Galerkin approximation of~\eqref{eq:newnonlinear}, i.e.
\begin{equation*}
\dprod{\F(\rec{n+1})-\F(\un{n+1}),v}=0 \qquad \forall v \in X_N.
\end{equation*}
Then, with the aid of~(F2), we infer that
\begin{align*}
\nu\norm{\rec{n+1}-\un{n+1}}_X^2
&\le\dprod{\F(\rec{n+1})-\F(\un{n+1}),\rec{n+1}-\un{n+1}}.
\end{align*}
Hence,
\begin{align*}
\nu\norm{\rec{n+1}-\un{n+1}}_X
&\le\sup_{\stackrel{w\in X}{\|w\|_X=1}}\dprod{\F(\rec{n+1})-\F(\un{n+1}),w}\\
&\le\sup_{\stackrel{w\in X}{\|w\|_X=1}}\left\{(\psi_N(\un{n+1}),w)_X-\dprod{\F(\un{n+1}),w}\right\}.
\end{align*}
Now, suppose that there exists a computable bound $\eta(\un{n+1})$ such that
\[
\sup_{\stackrel{w\in X}{\|w\|_X=1}}\left\{(\psi_N(\un{n+1}),w)_X-\dprod{\F(\un{n+1}),w}\right\}
\le\eta(\un{n+1}).
\]
Then,
\begin{align} \label{eq:errestnonlin}
\norm{\rec{n+1}-\un{n+1}}_X \leq \nu^{-1}\eta(\un{n+1}).
\end{align}
Similarly as in the linear case, in the specific context of the finite element method for elliptic PDE, such residual bounds can be obtained by standard techniques, see, e.g.,~\cite{AinsworthOden:96,Verfurth:13}. This will be carried out in Section~\ref{sec:ILFEM} for quasilinear elliptic PDE. 

\begin{theorem} \label{thm:aposterioriNG}
Given~{\rm (F1)} and~{\rm (F2)} (cf.~Section~\ref{sec:abstract}), there holds the a posteriori error bound
\begin{align*}
 \norm{u^\star-\un{n+1}}_X \leq \frac{1}{\Fsm}\left(\eta(\un{n+1})+ \norm{\psi_N(\un{n+1})}_{X}\right),
\end{align*}
where $u^\star$ is the exact solution of \eqref{eq:F=0}.
\end{theorem}

We note that, in the bound above, $\Fsm^{-1}\eta(\un{n+1})$ is an indicator for the discretization error by a similar argument as in Remark~\ref{rem:linestimator}, and $\Fsm^{-1} \norm{\psi_N(\un{n+1})}_{H}$ controls the linearization error. Indeed, since $\psi_N(\u_N)=0$ and $\{\un{n}\}_{n \geq 0}$ converges to the solution $\u_N$ of~\eqref{eq:itNnonl} by our analysis in Section~\ref{sec:itlin}, we see that $\norm{\psi_N(\un{n+1})}_X\to 0$ as~$n\to\infty$.

\begin{proof}
 By invoking the triangle inequality and \eqref{eq:errestnonlin}, we find that
\[
  \norm{u^\star-\un{n+1}}_X \leq \nu^{-1}\eta(\un{n+1}) + \norm{u^\star-\rec{n+1}}_X.
\]
Moreover, due to (F2), we observe that 
\begin{align*}
 \Fsm \norm{u^\star-\rec{n+1}}^2_X &\leq \dprod{\F(u^\star)-\F(\rec{n+1}),u^\star-\rec{n+1}} = \dprod{\F(\rec{n+1}),\rec{n+1}-u^\star}.
 \end{align*}
By using~\eqref{eq:newnonlinear}, and upon applying the Cauchy-Schwarz inequality, this leads to
\[
 \Fsm \norm{u^\star-\rec{n+1}}^2_X \leq (\psi_N(\un{n+1}),\rec{n+1}-u^\star)_X  \leq \norm{\psi_N(\un{n+1})}_X \norm{\rec{n+1}-u^\star}_X.
\]
This yields the claim.
\end{proof}

\begin{remark}
We compare the a posteriori error estimators from Theorem~\ref{thm:aposteriorierror} and Theorem~\ref{thm:aposterioriNG}: The proof of Theorem~\ref{thm:aposteriorierror}---and thus the constants in the bound~\eqref{eq:linestimator}---strongly depend on the choice of the preconditioning operator $\A$, i.e. on the specific iterative linearization method. Moreover, the same comment applies for the computable bound $\eta(\un{n+1},\un{n})$. In contrast, the estimator from Theorem~\ref{thm:aposterioriNG}, resulting from the nonlinear elliptic construction, as well as the computable bound $\eta(\un{n})$ are completely independent of the iteration scheme, and merely relies on the underlying PDE problem. We note further that the a posteriori error estimator from Theorem~\ref{thm:aposteriorierror} allows for more general reflexive Banach spaces~$X$, whereas Theorem~\ref{thm:aposterioriNG} requires a Hilbert space setting.   
\end{remark}


\section{An abstract ILG procedure}\label{sec:adaptivealgorithm}

The estimates from Theorems~\ref{thm:aposteriorierror} and~\ref{thm:aposterioriNG} allow to control the error between the solution of~\eqref{eq:F=0} and the discrete system~\eqref{eq:itN1} with respect to two individual terms, one of which expresses the error of the linearization, and will be denoted by $\mathcal{E}_{{\rm Linear},N}^{n}$, and the other, which we signify by $\mathcal{E}_{{\rm Galerkin},N}^{n}$, bounds the Galerkin discretization error. In a finite element context, the latter error will typically be composed of local contributions for each element; this, in turn, enables to refine the mesh locally. The algorithm, which will be presented below, uses an \emph{adaptive interplay} between those two controlling terms. More precisely, on a given Galerkin space, we iterate as long as the linearization error dominates and, in addition, until it is, in a certain way, smaller than a given bound depending on~the number of Galerkin space enrichments, $N$. Once the linearization error is small enough, and is up to a factor $\vartheta$ less than the one arising from the Galerkin method, we enrich the Galerkin space according to the local error indicators in order to attain a smaller discretization error. Subsequently, we will perform the linearization on the enriched space. In this way, the goal of the ILG algorithm is to compute an approximation of the solution of \eqref{eq:F=0} which, on the one hand, is sufficiently accurate, and, on the other hand, is attained from a minimal number of iterations.

\subsection{Adaptive ILG algorithm}
For the purpose of this section, we assume that our ILG Algorithm~\ref{alg:abstractadaptivealgorithm}, to be presented below, generates a sequence of hierarchically enriched Galerkin spaces, $X_0 \subset X_1 \subset X_2 \subset \dots$, on each of which we perform at least one iterative step. Furthermore, we will make use of a prescribed positive function $\sigma:\mathbb{N} \to (0,\infty)$ which satisfies
\begin{align} \label{eq:sigmaproperty}
 \sigma(N) \to 0 \ \text{ for } \ N \to \infty.
\end{align}
Its role is to ensure that the linearization error tends to zero for an increasing number~$N$ of Galerkin space enrichments. For instance, in the context of the finite element method, a sensible choice is $\sigma(N)=\mathcal{O}(|\mathcal{T}_N|^{-s})$, where $|\mathcal{T}_N|$ signifies the number of elements in the mesh, and $s$ is the expected convergence rate; cf.~Section~\ref{sec:numexp} below. Recall that, for any fixed $N \geq 0$, our theory in Section~\ref{sec:itlin} guarantees, under certain conditions, that the difference $\norm{u_N^n-u_N^{n-1}}_X$ tends to zero for increasing $n$; in particular, it can be made smaller than $\sigma(N)$ for $n$ large enough. 

An adaptive ILG procedure for the interactive reduction of discretization and linearization errors is proposed in Algorithm~\ref{alg:abstractadaptivealgorithm}. We note that this algorithm can be performed with any of the iterative procedures from Section~\ref{sec:particulariterations}, and with either the error estimators obtained from the linear or nonlinear elliptic reconstructions from Section~\ref{sec:aposteriorianal}. The input and output arguments as well as the components of the implemented algorithm may, of course, depend on the error estimator and the specific iterative linearization scheme applied.

\begin{algorithm}
\caption{Adaptive ILG algorithm}
\label{alg:abstractadaptivealgorithm} 
\begin{algorithmic}[1]
\State Prescribe a tolerance $\tol>0$, and an adaptivity parameter $\vartheta>0$. Set $N:=0$ and $n:=0$. Start with an initial Galerkin space $X_0 \subset X$, and an initial guess $u^0_0 \in X_0$.
\Repeat
\State Set $\mathcal{E}_{{\rm Linear},N}^{n}:=1$ and $\mathcal{E}_{{\rm Galerkin},N}^n:=0$.
\While {$\mathcal{E}_{{\rm Galerkin},N}^n \leq \vartheta \mathcal{E}_{{\rm Linear},N}^{n}$ or $\norm{u_N^n-u_N^{n-1}}_X > \sigma(N)$}
\State Perform a single iterative linearization step to obtain~$u_N^{n+1}$ from~$u_N^{n}$; cf.~\eqref{eq:itN1}.
\State Estimate the linearization error $\mathcal{E}_{{\rm Linear},N}^{n+1}$ and the Galerkin error indicator $\mathcal{E}_{{\rm Galerkin},N}^{n+1}$.
\State Update~$n\gets n+1$.
\EndWhile
\State \multiline{Let $u^{n^\star}_N:=u_N^n \in X_N$, and enrich the Galerkin space $X_N$ appropriately based on the error indicator $\mathcal{E}_{{\rm Galerkin},N}^n$ in order to obtain $X_{N+1}$.}
\State Set $\mathcal{E}_{\rm{tot}}:=\mathcal{E}_{{\rm Linear},N}^{n}+\mathcal{E}_{{\rm Galerkin},N}^n$.
\State Define $u_{N+1}^0 := u^{n^\star}_N$ by inclusion $X_{N+1} \hookleftarrow X_N$.
\State Update $N\gets N+1$, and set~$n:=0$.
\Until {$\mathcal{E}_{{\rm tot}} < \tol$.}\\
\Return the sequence of discrete solutions $u^{n^\star}_N \in X_N$. 
\end{algorithmic}
\end{algorithm}

\subsection{A remark on convergence}
Given a Galerkin space~$X_N$, we recall the solution~$\u_N\in X_N$ of \eqref{eq:itNnonl}. Furthermore, we let, as in the Algorithm~\ref{alg:abstractadaptivealgorithm}, $u_N^{n^\star}\in X_N$ be the final approximation on $X_N$ (i.e.~before the Galerkin space is enriched). We establish the convergence of $u^{n^\star}_N$ to the unique solution $u^\star$ of \eqref{eq:F=0} under the following  assumption:
 \begin{enumerate}[(AG)]
  \item The hierarchically enriched Galerkin spaces $X_0 \subset X_1 \subset X_2\subset\ldots$ generated by Algorithm~\ref{alg:abstractadaptivealgorithm} are such that the iterative Galerkin approximations $\u_N \in X_N$ from~\eqref{eq:itNnonl} converge to the exact solution $u^\star \in X $ of \eqref{eq:F=0} for ~$N\to\infty$. 
 \end{enumerate}
 
 \begin{proposition}
If $\F$ from~\eqref{eq:F=0} fulfills~{\rm (F1)} and~{\rm (F2)} (cf.~Section~\ref{sec:abstract}), and the Galerkin method satisfies~{\rm (AG)}, then Algorithm~\ref{alg:abstractadaptivealgorithm}based on an iterative linearization scheme~\eqref{eq:itN1}, with $a(u;\cdot,\cdot)$ satisfying the properties~\eqref{eq:coercive} and~\eqref{eq:continuity}, generates a sequence $\{u_N^{n^\star}\}_{N \geq 0}$ which converges to the unique solution $u^\star \in X$ of \eqref{eq:F=0}.
 \end{proposition}

\begin{proof}
Using (F2) and involving \eqref{eq:itNnonl}, it holds that
 \begin{align*}
  \Fsm \norm{\u_N-u_N^{n^\star-1}}_X^2 &\leq \dprod{\F(\u_N)-\F(u_N^{n^\star-1}),\u_N-u_N^{n^\star-1}} = \dprod{\F(u_N^{n^\star-1}),u_N^{n^\star-1}-\u_N}.
 \end{align*}
Invoking \eqref{eq:f(u)}, \eqref{eq:itN1}, and \eqref{eq:continuity}, we obtain
\begin{align*}
 \Fsm \norm{\u_N-\un{n^\star-1}}_X^2 \leq a(\un{n^\star-1};\un{n^\star-1}-\un{n^\star},\un{n^\star-1}-\u_N) \leq \beta \norm{\un{n^\star}-\un{n^\star-1}}_X \norm{\un{n^\star-1}-\u_N}_X,
\end{align*}
and thus
\begin{align*} 
\norm{\u_N-u_N^{n^\star-1}}_X \leq \frac{\beta}{\Fsm} \norm{\un{n^\star}-\un{n^\star-1}}_X.
\end{align*}
By the triangle inequality, this leads to 
\begin{align}\label{eq:bounduNun}
 \norm{\u_N-\un{n^\star}}_X \leq \norm{\u_N-\un{n^\star-1}}_X+\norm{\un{n^\star}-\un{n^\star-1}}_X \leq \left(\frac{\beta}{\Fsm}+1\right)\norm{\un{n^\star}-\un{n^\star-1}}_X.
\end{align}
Notice that the stopping criterion for the while loop in Algorithm~\ref{alg:abstractadaptivealgorithm} implies that
 \begin{align} \label{eq:assdifference}
  \norm{\un{n^*}-\un{n^*-1}}_X \leq \sigma(N) \qquad \forall N \geq 0,
 \end{align}
 with $\sigma$ satisfying~\eqref{eq:sigmaproperty}. Then, invoking the triangle inequality, as well as~\eqref{eq:bounduNun} and~\eqref{eq:assdifference}, yields \begin{align*}
  \norm{u^\star-\un{n^\star}}_X \leq \norm{u^\star-\u_N}_X+\norm{\u_N-\un{n^\star}}_X \leq \norm{u^\star-\u_N}_X+\left(\frac{\beta}{\Fsm}+1\right)\sigma(N).
 \end{align*}
The first term on the right-hand side tends to zero for $N \to \infty$ by virtue of~(AG), and the same holds true for the second term due to~\eqref{eq:sigmaproperty}. We deduce that $u_N^{n^\star} \to u^\star$ as $N \to \infty$. 
\end{proof}

In our subsequent paper \cite{HeidWihler:19} we further analyze the convergence of the adaptive ILG algorithm. In fact, we establish the linear convergence rate of our algorithm, with a slightly different a posteriori error estimator, under reasonable assumptions. For instance, in the context of finite element discretizations of second-order PDE in divergence form, cf.~Section~\ref{sec:scl}, we state in \cite{HeidWihler:19} an a posteriori error estimator which guarantees the linear convergence regime.

\section{Application to second-order PDE in divergence form} \label{sec:scl}

In this section, we will apply our analytical findings to the quasilinear elliptic PDE problem
\begin{align} \label{eq:operatorscl}
u\in X:\qquad \F(u):= - \nabla \cdot \left\{\mu\left(\left|\nabla u\right|^2\right) \nabla{u}\right\}-g=0\qquad\text{in }X'.
\end{align}
Here, $\Omega \subset \mathbb{R}^d$, for $d \in \mathbb{N}$, is an open and bounded domain with Lipschitz boundary $\Gamma:= \partial \Omega$, and $X:=H_0^1(\Omega)$ is the standard Sobolev space of $H^1$-functions on $\Omega$ with zero trace along the boundary $\Gamma$; the inner product and norm on~$X$ are defined, respectively, by~$(u,v)_X:=(\nabla u,\nabla v)_{L^2(\Omega)}$ and~$\norm{u}_X:=\|\nabla u\|_{L^2(\Omega)}$, for $u,v\in X$. Equations of the form~\eqref{eq:operatorscl} are widely used in mathematical models of physical applications including, for instance, hydro- and gas-dynamics, or plasticity; we refer to~\cite[\S69.2--69.3] {Zeidler:88} and~\cite[\S1.1]{Astala:09} for a discussion of the physical meaning. We suppose that $g \in X'=H^{-1}(\Omega)$ in~\eqref{eq:operatorscl} is given, and $\mu \in C^1([0,\infty))$ fulfills
\begin{align} \label{en:assmu}
m_\mu(t-s) \leq \mu(t^2)t-\mu(s^2)s \leq M_\mu (t-s), \qquad t \geq s \geq 0,
\end{align}
with constants $m_\mu, M_\mu>0$. In particular, upon setting $s=0$, we observe that 
\begin{align} \label{eq:mubound}
 m_\mu \leq \mu(t^2) \leq M_\mu \qquad \forall t \geq 0.
\end{align}
Under condition~\eqref{en:assmu} it can be shown that the nonlinear operator~$\F$ from~\eqref{eq:operatorscl} satisfies the properties~(F1) and~(F2) with
\begin{align}\label{eq:nuL}
\nu=m_{\mu},\qquad\Flc=3M_{\mu};
\end{align} 
see~\cite[Proposition~25.26]{Zeidler:90}.

We note the weak form of the boundary value problem~\eqref{eq:operatorscl} in $X$:
\begin{align}\label{eq:sclweak}
u\in X:\qquad \int_\Omega \muf{u} \nabla u \cdot \nabla v \dx = \dprod{g,v} \qquad \forall v \in X.
\end{align}

\subsection{Convergence of iterative linearizations}\label{sec:itscl}

In the sequel, we will investigate the convergence of the various iteration schemes from Section~\ref{sec:particulariterations} as applied to the PDE~\eqref{eq:operatorscl}. The convergence of the Zarantonello iteration follows immediately from Theorem~\ref{thm:zarantonello}.

\begin{proposition}
If $\mu$ satisfies~\eqref{en:assmu} and $\F$ is given by \eqref{eq:operatorscl}, then the Zarantonello iteration \eqref{eq:zarantonelloit}, i.e.
\[
u^{n+1}\in X:\qquad-\Delta u^{n+1}=-\Delta u^n+\delta\nabla \cdot \left\{\mu\left(\left|\nabla u^n\right|^2\right) \nabla{u^n}\right\}+\delta g,\qquad n\ge 0,
\] 
converges to the unique solution of \eqref{eq:operatorscl} for any $\delta \in \left]0,\nicefrac{2 m_\mu}{9 M^2_\mu}\right[$.
\end{proposition}

In order to study the Ka\v{c}anov iteration method for~\eqref{eq:operatorscl}, let us define, for~$u\in X$, the linear preconditioning operator 
\begin{align} \label{eq:kacanovobjects}
\A(u)v:=-\nabla \cdot \left\{\muf{u} \nabla v\right\},\qquad v\in X.
\end{align}
In addition to~\eqref{en:assmu}, we assume that $\mu$ is monotone decreasing, i.e.
\begin{equation}\label{en:assmu2}
\mu'(t) \leq 0\qquad\forall t \geq 0.
\end{equation} 

\begin{proposition}
Let $\mu$ satisfy~\eqref{en:assmu} and \eqref{en:assmu2}. Then, the Ka\v{c}anov iteration~\eqref{eq:kacanovstrong}, i.e.
\[
u^{n+1}\in X:\qquad-\nabla \cdot \left\{\muf{u^n} \nabla u^{n+1}\right\}=g,\qquad n\ge 0,
\]
converges to the unique solution of \eqref{eq:operatorscl}.
\end{proposition}

\begin{proof}
We will show that the assumptions of Theorem~\ref{thm:kacanov} are satisfied. To this end, for~$\A$ from~\eqref{eq:kacanovobjects}, and any~$u\in X$, we define the symmetric bilinear form~$a(u;v,w):=\dprod{\A(u)v,w}$, for $v,w\in X$. Then, using~\eqref{eq:mubound} in combination with the Cauchy-Schwarz inequality shows the coercivity and continuity properties~\eqref{eq:coercive} and~\eqref{eq:continuity} with
\begin{equation}\label{eq:Kab}
\alpha=m_\mu,\qquad \beta=M_\mu,
\end{equation}
respectively.
Furthermore, we introduce the potential $\G:X \to  \mathbb{R}$ by 
\begin{equation}\label{eq:psi}
\G(u):=\int_\Omega \psi\left(\left|\nabla u\right|^2\right) \dx,\qquad
\text{with}\quad 
\psi(s):=\frac12\int_0^s \mu(t) \dt. 
\end{equation}
For~$u\in X$, taking the Gateaux derivative of $\G$, we find that
\begin{align*}
\dprod{\G'(u),v}&=\int_\Omega 2 \psi'\left(\left|\nabla u\right|^2\right)\nabla u \cdot \nabla v \dx =\int_\Omega \muf{u} \nabla u \cdot \nabla v \dx =a(u;u,v),
\end{align*}
for any~$v\in X$. Thus, we infer that $\G'(u)=a(u;u,\cdot)=\F(u)+g$. Recalling~(F2),  this implies the strong monotonicity property~\eqref{eq:Gmono} with~$c_0=m_\mu$, and we conclude that (K1) holds true. In addition, due to~\eqref{en:assmu2}, 
for any $t\ge s\ge 0$, it holds that 
\begin{align*}
\psi(t)-\psi(s)
=\frac12 \int_s^t \mu(\tau) \, \d \tau \geq \frac12 (t-s)\mu(t),
\end{align*}
and similarly for~$s\ge t\ge 0$,
\begin{align*}
\psi(t)-\psi(s)
=-\frac12 \int_t^s \mu(\tau) \, \d \tau \geq -\frac12 (s-t)\mu(t)
=\frac12(t-s)\mu(t).
\end{align*}
Hence, for any~$u,v\in X$, we have
\begin{align*}
\G(u)-\G(v) 
&\geq \frac12 \int_\Omega \muf{u} \left(\left|\nabla u\right|^2-\left| \nabla v \right|^2\right) \dx 
= \frac12 \left(a(u;u,u)-a(u;v,v)\right),
\end{align*}
which shows (K2).
\end{proof}

Finally, we turn our attention to the damped Newton iteration. 

\begin{proposition} \label{prop:newton}
Let $\mu$ satisfy \eqref{en:assmu} and \eqref{en:assmu2}. Moreover, suppose that the damping parameter $\delta:\,X\to[\delta_{\min},\delta_{\max}]$ is a continuous functional, for some constants~$\delta_{\min},\delta_{\max}$, with~$0<\delta_{\min}\le\delta_{\max}<\nicefrac{2 m_\mu}{3 M_\mu}$. Then, the damped Newton iteration~\eqref{eq:newtonstrong} for the nonlinear PDE~\eqref{eq:operatorscl} converges to its unique solution in~$X$.
\end{proposition}

We will prove this proposition by showing that the assumptions of Theorem~\ref{thm:newton} are satisfied. For this purpose we require the following auxiliary result.

\begin{lemma}\label{lemma:munewton}
If $\mu$ satisfies \eqref{en:assmu}, then the operator~$u\mapsto\F'(u)u$ is continuous from~$X$ to~$X'$ with respect to the weak topology on~$X'$.
\end{lemma}

\begin{proof}
By taking the limit~$s\nearrow t$ in \eqref{en:assmu}, we infer that $m_\mu \leq \frac{\mathrm{d}}{\dt}\left(\mu(t^2)t\right) \leq M_\mu$, and, thereby, 
\begin{equation}\label{eq:auxlem}
m_\mu \leq 2 \mu'(t^2)t^2+\mu(t^2) \leq M_\mu\qquad\forall t\ge 0.
\end{equation}
Moreover, a simple but lengthy calculation shows that 
\begin{equation}\label{eq:rep}
\dprod{\F'(u)v,w}=\int_{\Omega} 2 \mu'(|\nabla u|^2)(\nabla u \cdot \nabla v)(\nabla u \cdot \nabla w) \dx + \int_{\Omega} \mu(|\nabla u|^2)\nabla v \cdot \nabla w \dx,
\end{equation}
for any~$u,v,w\in X$. Consider a sequence~$\{u^k\}_{k\ge 0}\subset X$ which converges to a limit~$u\in X$, i.e.
\begin{equation}\label{eq:cp1}
\norm{u-u^k}_X\to 0, \qquad k\to\infty.
\end{equation} 
Since $X=H_0^1(\Omega)$, we find that $\nabla u^k \to \nabla u$ in $L^2(\Omega)$ for $k \to \infty$. Thus, there is a subsequence such that
\begin{equation}\label{eq:cp2}
\nabla u^{k'}\to \nabla u\quad\text{a.e.~in~$\Omega$ for }k'\to\infty, 
\end{equation}
see, e.g., \cite[Theorem 3.12]{Rudin:87}.
Hence, defining the function~$\omega(t):=2\mu'(t)t+\mu(t)$, $t\ge 0$, it holds
\begin{align*}
\dprod{\F'(u)u-\F'(u^{k'})u^{k'},w}
&=\int_\Omega\left(\omega(|\nabla u|^2)-\omega(|\nabla u^{k'}|^2)\right)\nabla u\cdot\nabla w\dx\\
&\quad+\int_\Omega\omega(|\nabla u^{k'}|^2)\nabla(u-u^{k'})\cdot\nabla w\dx.
\end{align*}
We note that both terms on the right-hand side tend to~$0$ as~$k'\to\infty$: Indeed, for the first integral this follows from the continuity of~$\omega$, \eqref{eq:cp2}, \eqref{eq:auxlem}, and the dominated convergence theorem; for the second integral, we recall~\eqref{eq:auxlem} and~\eqref{eq:cp1}. Finally, referring to~\cite[Proposition~10.13(2)]{Zeidler:86}, we conclude the weak convergence of the entire sequence, i.e. $\F'(u^{k})u^{k}\rightharpoonup\F'(u)u$ as~$k\to\infty$. This finishes the proof.
\end{proof}

\begin{proof}[Proof of Proposition~\ref{prop:newton}]
For~$v=w$ in \eqref{eq:rep} we have that
\[
 \dprod{\F'(u)v,v}=\int_{\Omega} 2 \mu'(|\nabla u|^2)|\nabla u \cdot \nabla v|^2 + \int_{\Omega} \mu(|\nabla u|^2)|\nabla v|^2 \dx.
\]
Exploiting~\eqref{en:assmu2}, and using the Cauchy-Schwarz inequality, we notice that
\begin{align*}
 2 \mu'(|\nabla u|^2) |\nabla u \cdot \nabla v|^2 \geq 2 \mu'(|\nabla u|^2) |\nabla u|^2 |\nabla v|^2.
\end{align*}
It follows that
\begin{align*}
\dprod{\F'(u)v,v}
 &\ge \int_\Omega \left(2 \mu'(|\nabla u|^2) |\nabla u|^2+\muf{u}\right) |\nabla v|^2 \dx.
\end{align*}
Applying~\eqref{eq:auxlem} implies that $\dprod{\F'(u)v,v}
\geq m_\mu \norm{v}_X^2$ for any~$u,v\in X$; this shows~\eqref{eq:N11} with
\begin{equation}\label{eq:Na}
\fpco=m_\mu.
\end{equation}
In addition, in view of~\eqref{eq:nuL}, we observe that~$\nicefrac{2\fpco}{\Flc}=\nicefrac{2m_\mu}{3M_\mu}>\delta_{\max}$, as required in Theorem~\ref{thm:newton}.
Furthermore, application of the Cauchy-Schwarz inequality, and involving~\eqref{en:assmu2}, yields 
\begin{align*}
 \dprod{\F'(u)v,w} 
 &\leq \int_\Omega \left(\left|2 \mu'(|\nabla u|^2)\right| \left|\nabla u\right|^2+\muf{u}\right) |\nabla v||\nabla w| \dx\\
  &= -\int_\Omega \left(2 \mu'(|\nabla u|^2) \left|\nabla u\right|^2+\muf{u}\right) |\nabla v||\nabla w| \dx\\
&\quad  +2\int_\Omega \muf{u} |\nabla v||\nabla w| \dx.
\end{align*}
Employing~\eqref{eq:auxlem} and~\eqref{eq:mubound}, this leads to
\begin{align*}
 \dprod{\F'(u)v,w} &\leq (2 M_\mu-m_\mu) \int_\Omega |\nabla v| |\nabla w| \dx \leq (2 M_\mu-m_\mu) \norm{v}_X \norm{w}_X,
\end{align*}
which gives~\eqref{eq:N12} with
\begin{equation}\label{eq:Nb}
\fpbd=2M_\mu-m_\mu.
\end{equation} 
In order to prove~(N3), let us define the functional $\operator{H}:X \to  \mathbb{R}$ by 
\begin{align*}
\operator{H}(u):=\int_\Omega \psi\left(\left|\nabla u\right|^2\right) \dx-\dprod{g,u},\qquad u\in X,
\end{align*}
with~$\psi$ as in~\eqref{eq:psi}. It holds that
\begin{align*}
\dprod{\operator{H}'(u),v}=\int_\Omega \muf{u} \nabla u \cdot \nabla v\dx-\dprod{g,v} =\dprod{\F(u),v},
\end{align*}
for all $v \in X$. Finally, to establish~(N2), we introduce the functional $\G:X \to \mathbb{R}$ by $\G(u):=\mathcal{F}(u)-\operator{H}(u)$, where $\mathcal{F}(u):=\dprod{\F(u),u}$, $u\in X$. For~$u\in X$, the Gateaux derivative of $\mathcal{F}$ is given by
\begin{align*}
 \dprod{\mathcal{F}'(u),v}
 &= \dprod{\F'(u)u,v}+\dprod{\F(u),v}\qquad\forall v\in X.
\end{align*}
It follows that $\G'(u)=\mathcal{F}'(u)-\operator{H}'(u)=\F'(u)u+\F(u)-\F(u)=\F'(u)u$. Finally, due to Lemma~\ref{lemma:munewton} the mapping $u \mapsto \G'(u)=\F'(u)u$ is continuous with respect to the weak topology on $X'$. 
\end{proof}

\subsection{Iterative linearized FEM}\label{sec:ILFEM}

For the sake of discretizing \eqref{eq:sclweak}, and thereby, of obtaining an ILG formulation for~\eqref{eq:operatorscl}, we will use a conforming finite element framework. To illustrate our approach we deal with a physical domain $\Omega \subset \mathbb{R}^2$; we remark, however, that the discussion below can, in principle, be generalized to higher dimensions. We consider regular and shape-regular meshes $\mathcal{T}_h$ that partition the domain~$\Omega$ into open and disjoint triangles~$K\in\mathcal{T}_h$ such that $\overline{\Omega}=\bigcup_{K \in \mathcal{T}_h} \overline K$.
We denote by $h_K:=\diam{K}$ the diameter of $K \in \mathcal{T}_h$, and let $h:=\max_{K \in \mathcal{T}_h} h_K$. Moreover, we consider the finite element space 
\begin{align}\label{eq:Xh}
X_h:=\left\{v \in H^1_0(\Omega): v|_K \in \mathcal{P}_p(K) \ \forall K \in \mathcal{T}_h\right\},
\end{align}
where, for fixed~$p\in\mathbb{N}$, we signify by $\mathcal{P}_p(K)$ the space of all polynomials of total degree at most $p \geq 1$ on $K\in\mathcal{T}_h$. 

Within the adaptive ILG framework, we will consider a sequence of meshes~$\{\mathcal{T}_N\}_{N\ge0}$, whereby we start with an initial conforming triangulation~$\mathcal{T}_0$ of $\Omega$. All subsequent meshes are obtained by refinement, i.e. for~$N\ge 0$, the mesh $\mathcal{T}_{N+1}$ is a hierarchical refinement of $\mathcal{T}_N$. Moreover, we will denote by $X_N$ the finite element space associated to the mesh $\mathcal{T}_N$.

For an edge $e \subset \partial K^+ \cap \partial K^{-}$, which is the intersection of two neighbouring elements $K^{\pm} \in \mathcal{T}_N$, we signify by $\jmp{\bm v}|_e=\bm{v}^{+}|_e \cdot \bm{n}_{K^+}+\bm{v}^{-}|_e\cdot \bm{n}_{K^{-}}$ the jump of a (vector-valued) function $\bm{v}$ along~$e$, where $\bm{v}^{\pm}|_e$ denote the traces of the function $\bm{v}$ on the edge $e$ taken from the interior of $K^{\pm}$, respectively, and $\bm{n}_{K^{\pm}}$ are the unit outward normal vectors on $\partial K^{\pm}$, respectively.

\subsubsection{A posteriori error analysis via linear elliptic reconstruction}

In this section, we discuss the a posteriori error estimate from Theorem~\ref{thm:aposteriorierror} in the specific context of the nonlinear PDE~\eqref{eq:operatorscl} and the finite element framework presented above. Introducing the residual
\[
\R(u;v,w):=a(u;v,w)-\dprod{f(u),w},\qquad u,v\in X_N,\, w\in X,
\]
it is fairly straightforward to verify that, for all of the three iterative linearization schemes from Section~\ref{sec:itscl}, and for~$g\in L^2(\Omega)$ in~\eqref{eq:operatorscl}, it holds the special form
\[
\R(\un{n};\un{n+1},w)=-\int_{\Omega}\mathbf{q}^n_N \cdot \nabla w \dx+ \int_{\Omega}p^n_N w \dx \qquad\forall w \in X,
\]
with some $p^n_N \in L^2(\Omega)$ and $\mathbf{q}^n_N \in H^1(\Omega)^2$, which can be represented explicitly.
Then, recalling~\eqref{eq:itN1}, we may conclude that
\[
\R(\un{n};\un{n+1},w)
=\R(\un{n};\un{n+1},w-w_N)
=-\int_{\Omega}\mathbf{q}^n_N \cdot \nabla (w-w_N) \dx+ \int_{\Omega}p^n_N (w-w_N) \dx,
\]
for any~$w_N\in X_N$. Therefore, choosing~$w_N$ to be a quasi-interpolant of~$w$, and pursuing a standard residual-based a posteriori error analysis (see, e.g., \cite{Verfurth:13}), we deduce the upper bound
\begin{align*}
\sup_{\substack{w \in X \\ \norm{w}_X=1}} \R(\un{n};\un{n+1},w) \leq \coc\left(\sum_{K \in \mathcal{T}_N} \eta_K^2\right)^{\nicefrac{1}{2}},
\end{align*}
where~$\coc>0$ is an interpolation constant (only depending on the polynomial degree~$p$ and on the shape-regularity of the mesh), and
\begin{align}
\eta_K^2&= h_K^{2} \twon{ \nabla \cdot \mathbf{q}^n_N+ p^n_N}{K}^2 +  \frac{1}{2} h_K\twon{\jmp{\mathbf{q}^n_N}}{\partial K \setminus \Gamma}^2,\qquad K\in\mathcal{T}_N,\label{eq:eta1}
\end{align}
is a computable error indicator.

\begin{theorem} \label{thm:aposterioriLSCL}
Let $\F$ be defined by \eqref{eq:operatorscl} with $\mu$ fulfilling~\eqref{en:assmu} and \eqref{en:assmu2}, and let $X_N\subset H^1_0(\Omega)$ be a conforming finite element space as in~\eqref{eq:Xh} on a mesh~$\mathcal{T}_N$. If $u^\star$ is the unique solution of \eqref{eq:operatorscl}, and $\{\un{n}\}_{n \geq 0}$ is a sequence of ILG solutions obtained by any of the iterative linearization procedures from Section~\ref{sec:itscl} on $X_N$, then it holds the a posteriori estimate
\[
  \norm{u^\star-\un{n+1}}_X \leq \frac{\beta \coc}{\alpha m_\mu} \left(\sum_{K \in \mathcal{T}_N} \eta_K^2\right)^{\nicefrac{1}{2}} + \frac{\beta+3M_\mu}{m_\mu} \norm{\un{n+1}-\un{n}}_X,
\]
where $\coc>0$ is a constant, and
\[
(\alpha,\beta)=\begin{cases}
(\delta^{-1},\delta^{-1})&\text{for the Zarantonello iteration, cf.~\eqref{eq:Zab}},\\
(m_\mu,M_\mu)&\text{for the Ka\v{c}anov iteration, cf.~\eqref{eq:Kab}},\\
(\nicefrac{m_\mu}{\delta_{\max}},\nicefrac{(2M_\mu-m_\mu)}{\delta_{\min}})&\text{for the Newton iteration, cf.~\eqref{eq:Nab}, \eqref{eq:Na}, and~\eqref{eq:Nb}},\\
\end{cases}
\]
and~$\eta_K$, for~$K\in\mathcal{T}_N$, is defined in~\eqref{eq:eta1}.
\end{theorem}

\begin{proof}
The result follows from Theorem~\ref{thm:aposteriorierror}, whereby we replace the constants $\nu$ and~$\Flc$ from~\eqref{eq:nuL}, and insert the values of $\alpha$ and $\beta$ from~\eqref{eq:coercive} and~\eqref{eq:continuity} for the respective iterative schemes from Section~\ref{sec:itscl}.
\end{proof}

\subsubsection{Error estimator via \emph{nonlinear} elliptic reconstruction} 

Following our abstract analysis in Section~\ref{sec:aposterioriNGM}, we consider the residual
\[
\R(\un{n+1}):=\sup_{\stackrel{w\in X}{\|w\|_X=1}}\left\{(\psi_N(\un{n+1}),w)_X-\dprod{\F(\un{n+1}),w}\right\}.
\]
Noticing~\eqref{eq:R1}, for any~$w_N\in X_N$, we have
\[
\R(\un{n+1}):=\sup_{\stackrel{w\in X}{\|w\|_X=1}}\left\{(\psi_N(\un{n+1}),w-w_N)_X-\dprod{\F(\un{n+1}),w-w_N}\right\}.
\]
Then, for $g\in L^2(\Omega)$ in~\eqref{eq:operatorscl}, and~$w_N\in X_N$ an appropriate quasi-interpolant of~$w\in H^1_0(\Omega)$, we employ a standard residual-based a posteriori error analysis (see, e.g., \cite{Verfurth:13}) to infer the upper bound
\begin{align*}
\R(\un{n+1}) \leq \coc \left(\sum_{K \in \mathcal{T}_N} \eta_K^2\right)^{\nicefrac{1}{2}},
\end{align*}
where $\coc$ is a quasi-interpolation constant, and 
\begin{equation}\label{eq:eta2}
\begin{split}
\eta^2_K&=h_K^2 \twon{\Delta \psi_N(\un{n+1})+g+ \nabla \cdot \left\{\muf{\un{n+1}} \nabla \un{n+1}\right\}}{K}^2 \\
& \quad + \frac{1}{2} h_K \twon{ \jmp{\nabla \psi_N(\un{n+1})+\muf{\un{n+1}} \nabla \un{n+1}}}{\partial K \setminus \Gamma}^2,
\end{split}
\end{equation}
for any $K \in \mathcal{T}_N$. Then, invoking Theorem~\ref{thm:aposterioriNG} and recalling~\eqref{eq:nuL}, we obtain the following result.

\begin{theorem} \label{thm:nonlinearV1}
Given the same assumptions as in Theorem~\ref{thm:aposterioriLSCL}, then it holds the a posteriori error estimate
\begin{align*}
 \norm{u^\star-\un{n+1}}_X \leq \frac{\coc}{m_\mu} \left(\sum_{K \in \mathcal{T}_N} \eta_K^2\right)^{\nicefrac{1}{2}}+\frac{1}{m_\mu} \norm{\psi_N(\un{n+1})}_{L^2(\Omega)},
\end{align*}
where $u^\star$ is the unique solution of \eqref{eq:operatorscl}, $\coc$ is a constant,
and~$\eta_K$, for~$K\in\mathcal{T}_N$, is given in~\eqref{eq:eta2}. 
\end{theorem}

\subsection{Numerical Experiments} \label{sec:numexp}

In this section, we test the adaptive ILG Algorithm~\ref{alg:abstractadaptivealgorithm} in the context of the iterative linearized FEM for second-order PDE in divergence form discussed in Section~\ref{sec:scl}. We perform a series of numerical experiments to compare the various iterative linearization procedures from Section \ref{sec:particulariterations} and to validate the \textit{a posteriori} error estimators from Section~\ref{sec:ILFEM}. For all our experiments, we consider the L-shaped domain $\Omega=(-1,1)^2 \setminus ([0,1] \times [-1,0])$, and an initial mesh consisting of 192 uniform triangles. Moreover, we will always choose the initial guess to be $u^0 \equiv 0$, and run the algorithm until the number of elements exceeds $10^6$. On a given mesh, we perform at least one iterative linearization step, and continue until the linearization error is at most half as large as the discretization error, i.e.~we let $\vartheta=2$ in Algorithm~\ref{alg:abstractadaptivealgorithm}. Furthermore, for a given constant~$\Upsilon>0$, we let
\[
\sigma(N):=\Upsilon\left|\mathcal{T}_N\right|^{-\nicefrac12}\norm{u_0^1}_X, \qquad N \geq 0,
\]
which relates to the expected convergence rate of~$\mathcal{O}(|\mathcal{T}_N|^{-\nicefrac12})$; in our experiments below the choice~$\Upsilon=10$ has proved to be a sensible value. Moreover, we set the constant factors for the discretization and linearization estimators appearing in the right-hand sides of the a posteriori error bounds to~1 (cf.~Theorems~\ref{thm:aposterioriLSCL} and \ref{thm:nonlinearV1}). In the adaptive process, we mark the elements for refinement by use of the D\"orfler marking strategy, see \cite{Doerfler:96}, and process them by the newest vertex bisection method, see \cite{Mitchell:91}. The true error $\norm{u^\star-u_N^n}_X$ and the error estimator will be displayed each time before a mesh refinement is undertaken. Our implementation is based on the Matlab package~\cite{FunkenPraetoriusWissgott:11}, with the necessary modifications. 

In the Experiments~\ref{sec:smoothsolution}--\ref{sec:increasingdiffusion} below we consider the different iterative procedures discussed in Section~\ref{sec:particulariterations}. For the problems under consideration, our computations consistently indicate that, in the a posteriori error estimates from Theorem~\ref{thm:aposterioriLSCL} and Theorem~\ref{thm:nonlinearV1}, the discretization part  clearly dominantes the linearization contribution. Not surprisingly, after a brief initial mesh refinement phase, the algorithm only undertakes one iterative linearization step per space enrichment, i.e. our algorithm is highly efficient for the proposed examples. Moreover, both the discretization and linearization error indicators generally converge at the expected rate of $\mathcal{O}(|\mathcal{T}_N|^{-\nicefrac{1}{2}})$. More precisely, this holds true for any iterative scheme except for the damped Newton method (in combination with the a posteriori error estimator from Theorem~\ref{thm:nonlinearV1}), where the linearization error estimator exhibits an even higher convergence rate; this may result from the local quadratic convergence property of the Newton iteration.

\subsubsection{Smooth solution}\label{sec:smoothsolution}
We consider the nonlinear diffusion coefficient $\mu(t)=(t+1)^{-1}+\nicefrac{1}{2}$,
for $t \geq 0$, and select~$g$ in~\eqref{eq:operatorscl} such that the analytical solution of \eqref{eq:sclweak} is given by the smooth function
$u^\star(x,y)=\sin(\pi x) \sin(\pi y)$. It is straightforward to verify that $\mu$ fulfills the requirements \eqref{en:assmu} and \eqref{en:assmu2} from Section~\ref{sec:scl}, so that the convergence of the three iterative procedures from Section~\ref{sec:particulariterations} is guaranteed. The parameter $\dpa$ in the Zarantonello iteration~\eqref{eq:zarantonelloit} is chosen to be $0.85$ as this seems to be close to optimal. The initial damping parameter on the initial mesh for the damped Newton method is chosen to be $\dpa^{0}=1$ in Remark~\ref{rem:adnewton}; moreover, throughout all our experiments, the factor $\kappa$ for the correction and prediction strategy of the damping parameter is set to be $\nicefrac{1}{2}$. 

In Figure~\ref{fig:smoothsolution}, for each of the three iterative linearization schemes presented in Section~\ref{sec:itscl}, we plot the error $\norm{u^\star-u_N^n}_X$ and both error estimators from Theorems~\ref{thm:aposterioriLSCL} and~\ref{thm:nonlinearV1} against the number~$|\mathcal{T}_N|$ of elements in the mesh. In addition, we display the effectivity indices for each experiment, i.e. the ratio of the error estimator and the true error; we see that they are roughly bounded between 2 and~4. Furthermore, we notice that (nearly) optimal convergence rates $\mathcal{O}\left(|\mathcal{T}_N|^{-\nicefrac{1}{2}}\right)$ are achieved in all plots.

\begin{figure}[htb] 
 \subfloat[Zarantonello iteration with the a posteriori error bound from Theorem~\ref{thm:aposterioriLSCL}.]{\includegraphics[width=0.499\textwidth]{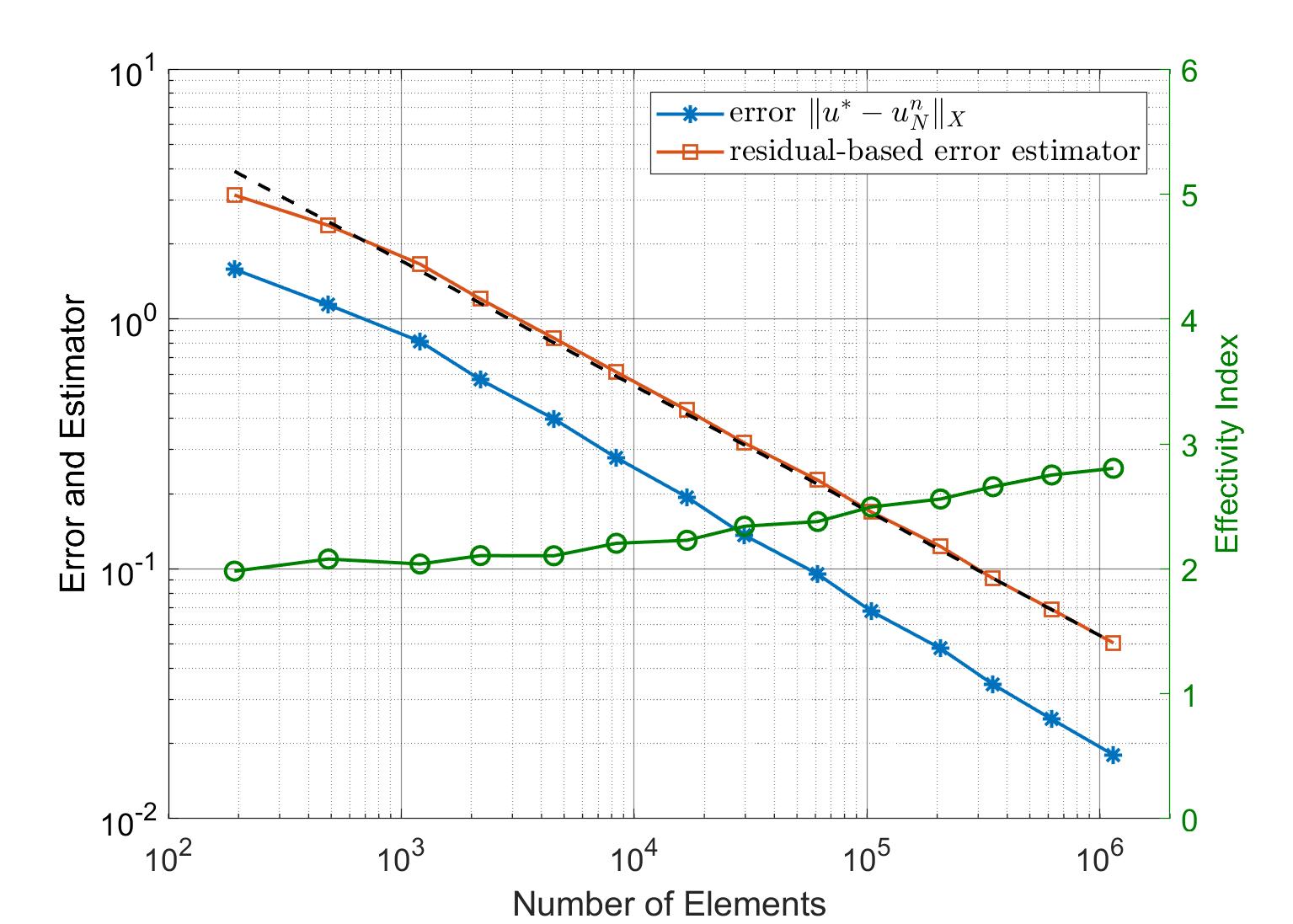}}\hfill
 \subfloat[Zarantonello iteration with the a posteriori error bound from Theorem~\ref{thm:nonlinearV1}.]{\includegraphics[width=0.499\textwidth]{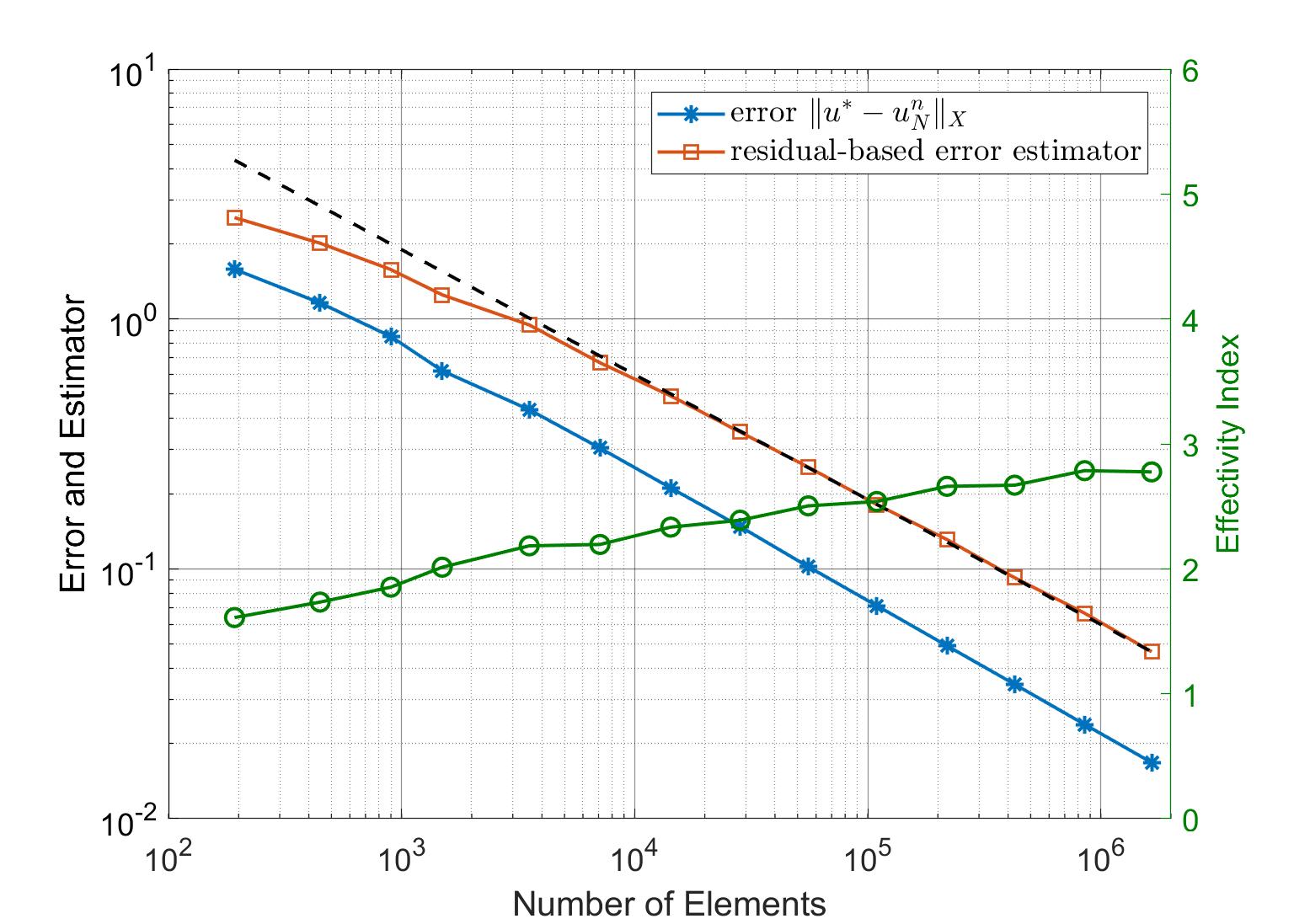}}\\[-1ex]
 \subfloat[Ka\v{c}anov iteration with the a posteriori error bound from Theorem~\ref{thm:aposterioriLSCL}.]{\includegraphics[width=0.499\textwidth]{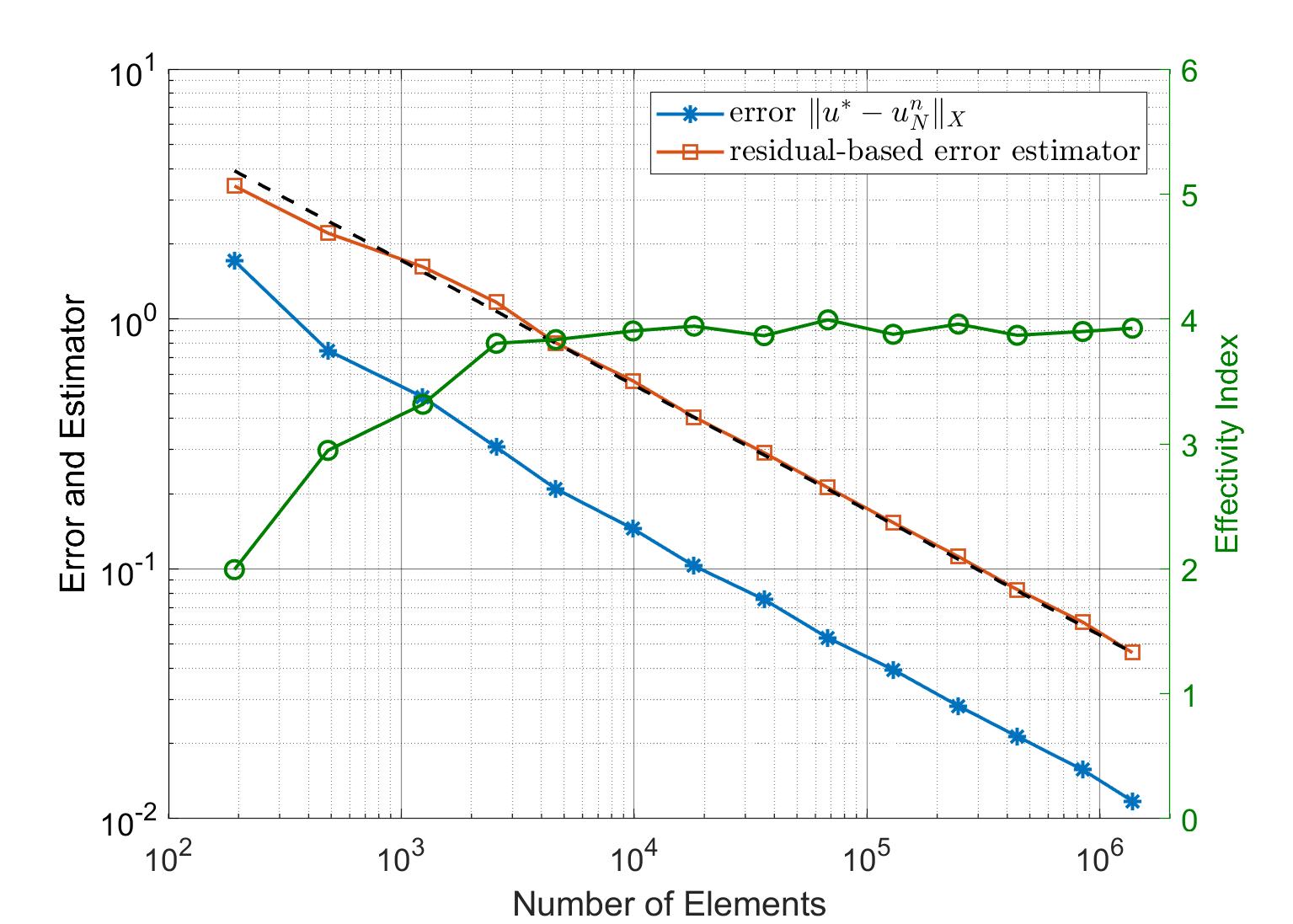}}\hfill 
 \subfloat[Ka\v{c}anov iteration with the a posteriori error bound from Theorem~\ref{thm:nonlinearV1}.]{\includegraphics[width=0.499\textwidth]{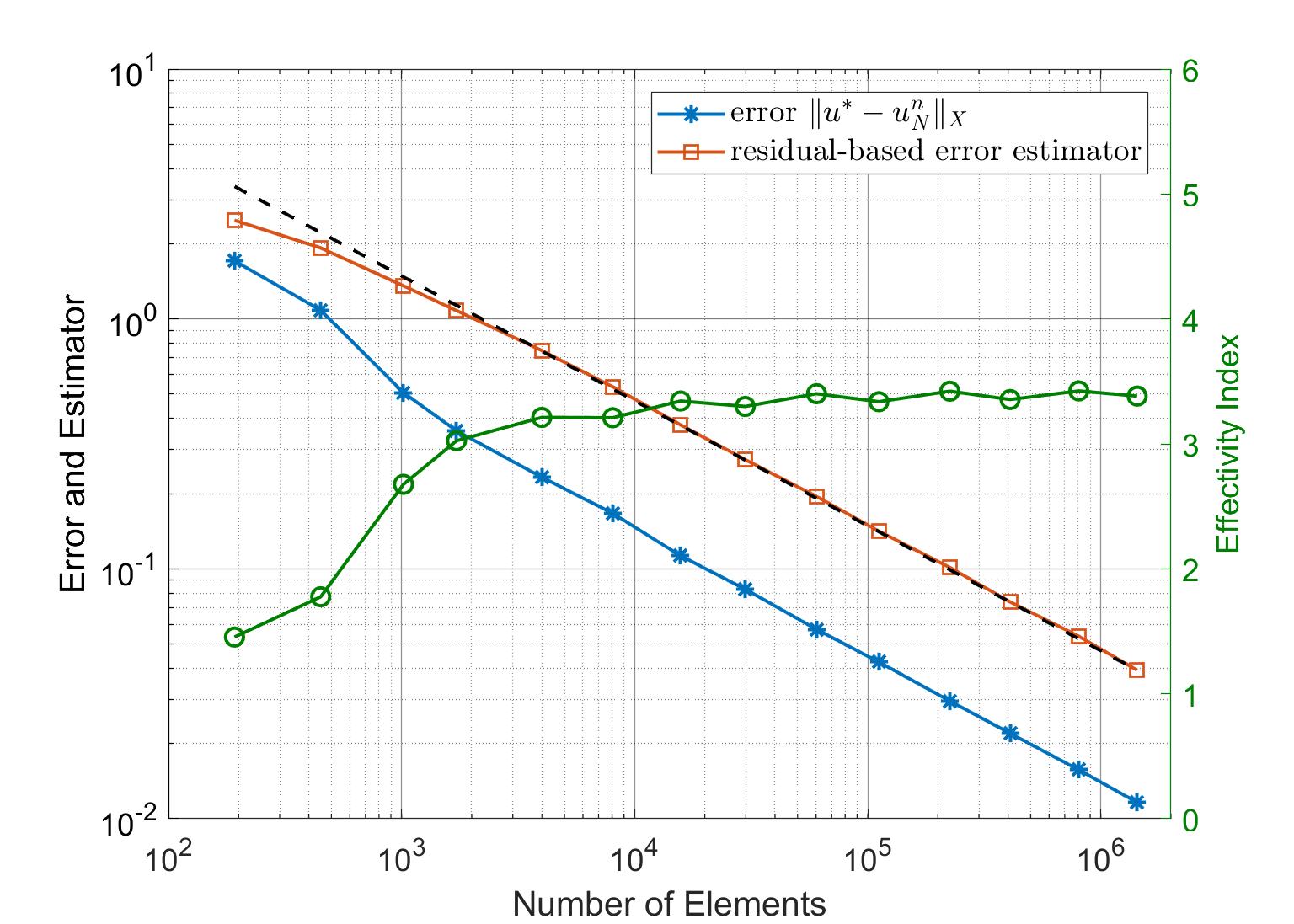}}\\[-1ex]
 \subfloat[Damped Newton iteration with the a posteriori error bound from Theorem~\ref{thm:aposterioriLSCL}.]{\includegraphics[width=0.499\textwidth]{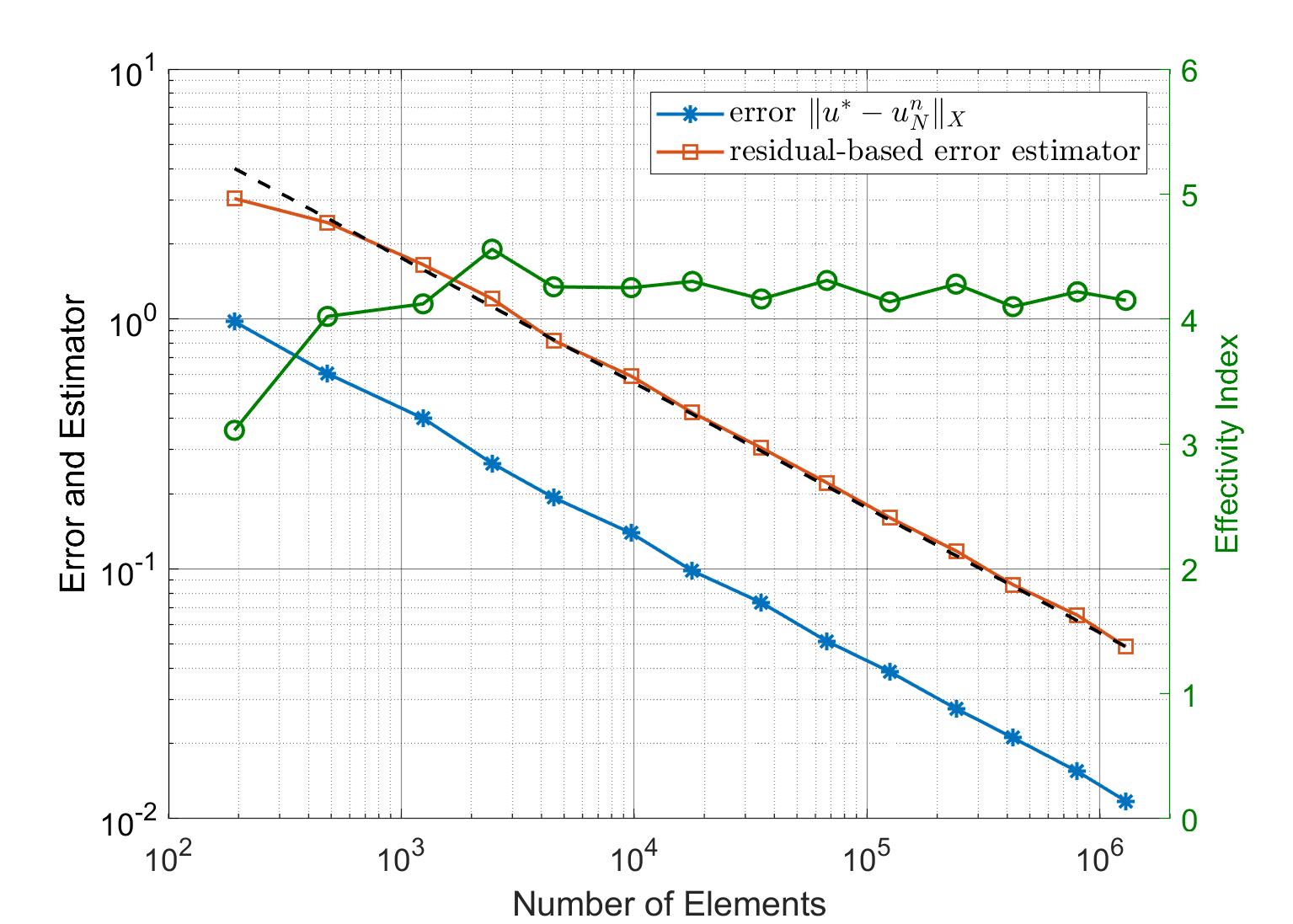}}\hfill
 \subfloat[Damped Newton iteration with the a posteriori error bound from Theorem~\ref{thm:nonlinearV1}.]{\includegraphics[width=0.499\textwidth]{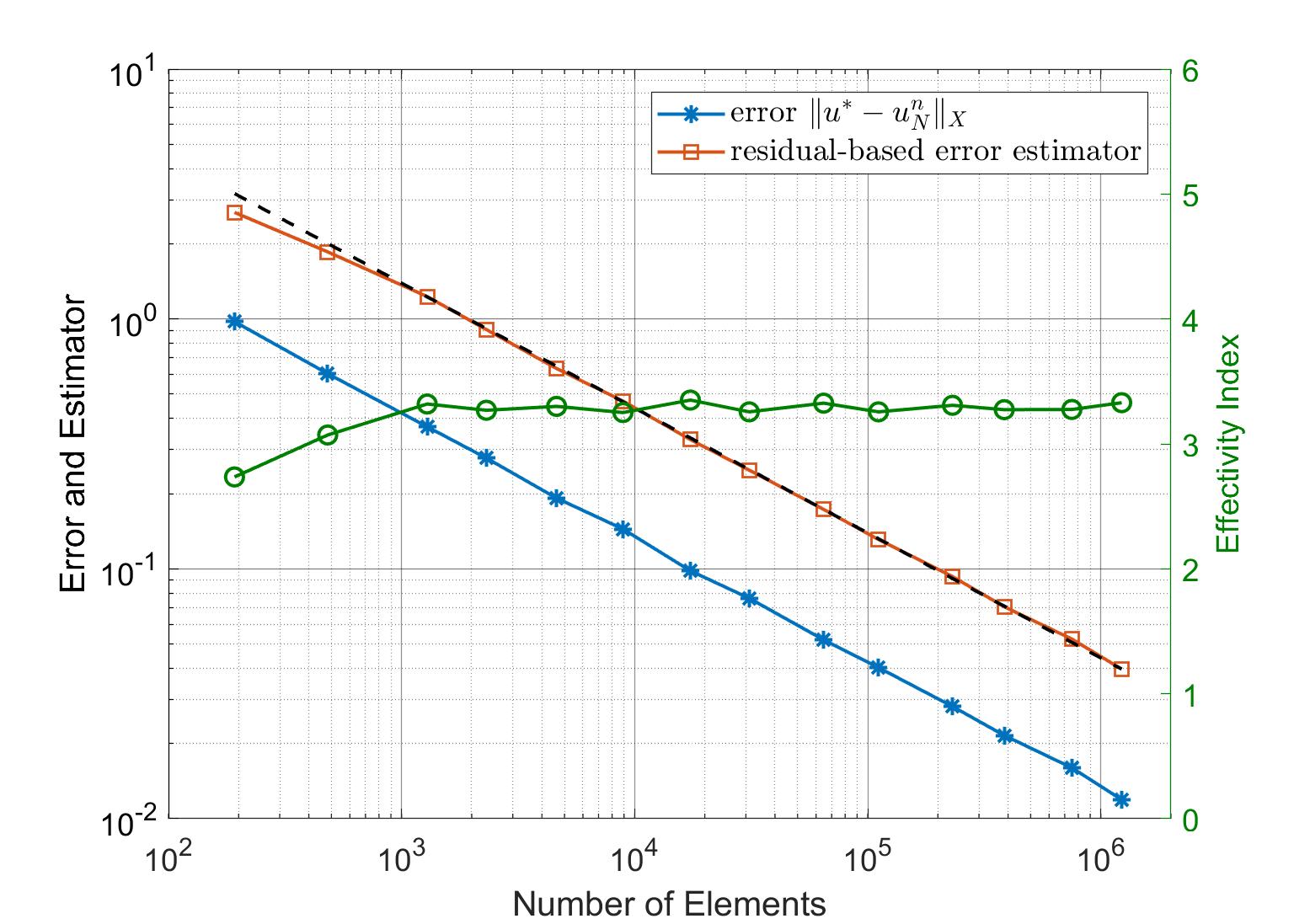}}
 \caption{Experiment~\ref{sec:smoothsolution}: Performance data for the error estimators from Theorem~\ref{thm:aposterioriLSCL} (left) and Theorem~\ref{thm:nonlinearV1} (right) for the Zarantonello, Ka\v{c}anov and Newton iterations.}
 \label{fig:smoothsolution}
\end{figure}

\subsubsection{Nonsmooth solution}\label{sec:nonsmoothsolution}
In our second experiment, we consider the nonlinear diffusion parameter $\mu(t)=1+\mathrm{e}^{-t}$, for $t \geq 0$. Again, it is easily seen that~$\mu$ satisfies the assumptions~\eqref{en:assmu} and~\eqref{en:assmu2}. We choose $g$ in \eqref{eq:operatorscl} such that the analytical solution is given by 
\begin{align} \label{eq:nonsmoothsolution}
u^\star(r,\varphi)=r^{\nicefrac{2}{3}}\sin\left(\nicefrac{2\varphi}{3}\right)(1-r \cos(\varphi))(1+r \cos(\varphi))(1- r \sin(\varphi))(1+r \sin(\varphi))\cos(\varphi),
\end{align}
where $r$ and $\varphi$ are polar coordinates. This is the prototype singularity for (linear) second-order elliptic problems with homogeneous Dirichlet boundary conditions in the L-shaped domain; in particular, we note that the gradient of~$u^\star$ is unbounded at the origin. As before, in Figure~\ref{fig:nonsmoothsolution}, we plot the error $\norm{u^\star-u_N^n}_X$, the error estimators from Theorems~\ref{thm:aposterioriLSCL} and~\ref{thm:nonlinearV1}, as well as the effectivity indices versus the number~$|\mathcal{T}_N|$ of elements in the mesh for each of the three iterative linearization schemes from Section~\ref{sec:ILFEM}. We let $\dpa=0.5$ for the Zarantonello iteration, and use the initial damping parameter $\dpa^{0}=1$ for the Newton method as in Experiment~\ref{sec:smoothsolution}. As before, we observe that optimal rates of convergence are attained in all six cases. 

\begin{figure}[htb] 
 \subfloat[Zarantonello iteration with the a posteriori error bound from Theorem~\ref{thm:aposterioriLSCL}.]{\includegraphics[width=0.499\textwidth]{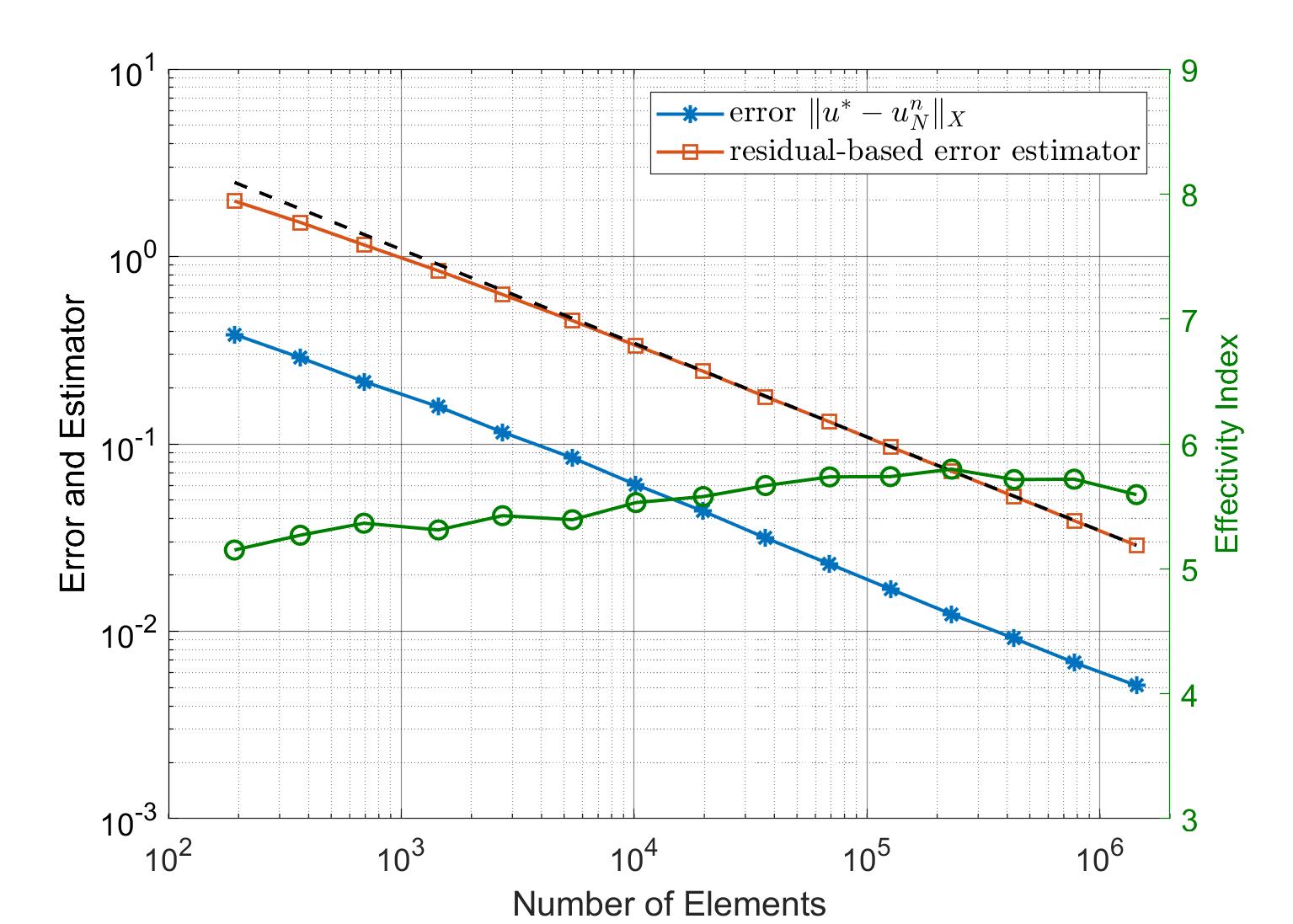}}\hfill
 \subfloat[Zarantonello iteration with the a posteriori error bound from Theorem~\ref{thm:nonlinearV1}.]{\includegraphics[width=0.499\textwidth]{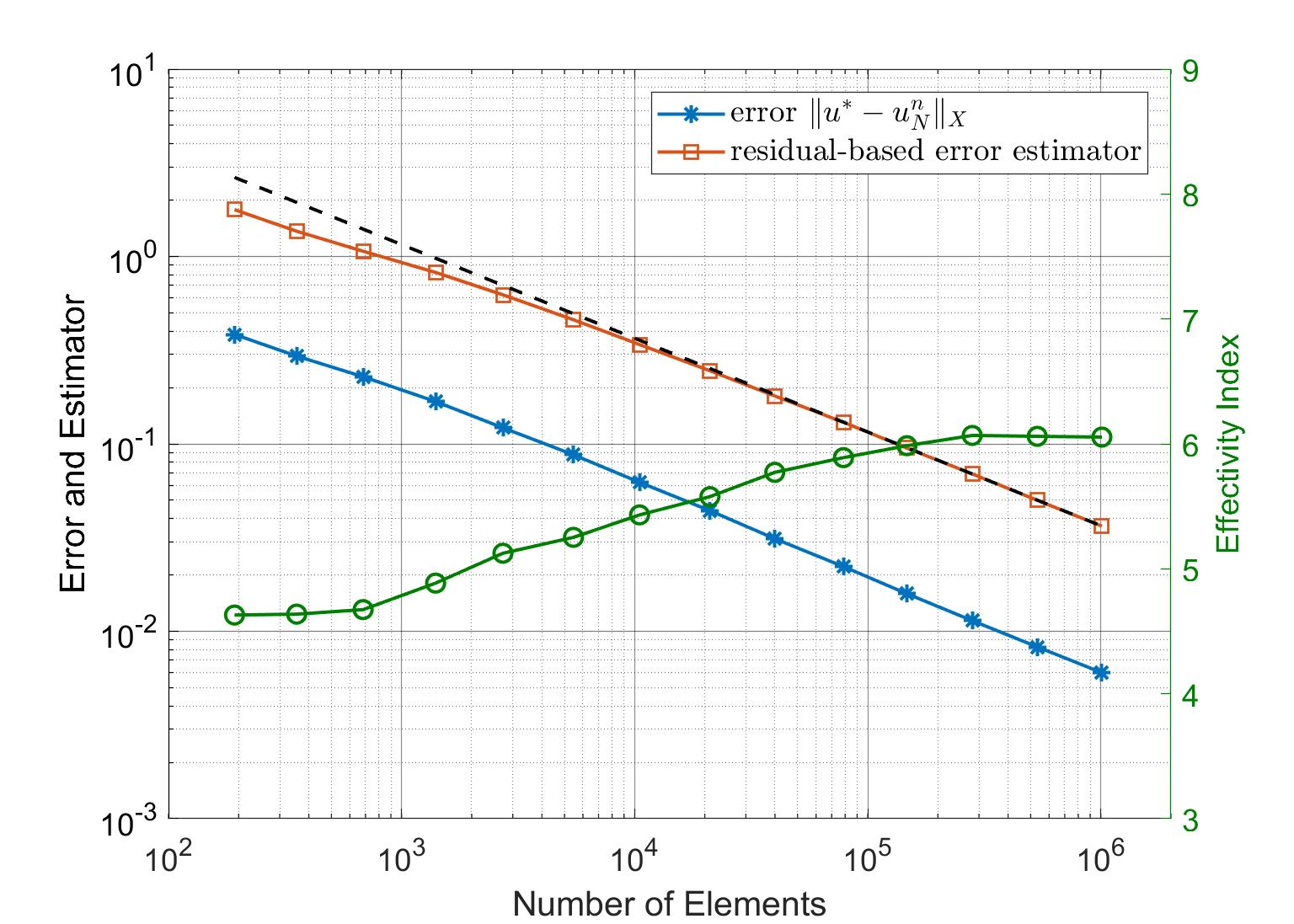}}\\[-1ex]
 \subfloat[Ka\v{c}anov iteration with the a posteriori error bound from Theorem~\ref{thm:aposterioriLSCL}.]{\includegraphics[width=0.499\textwidth]{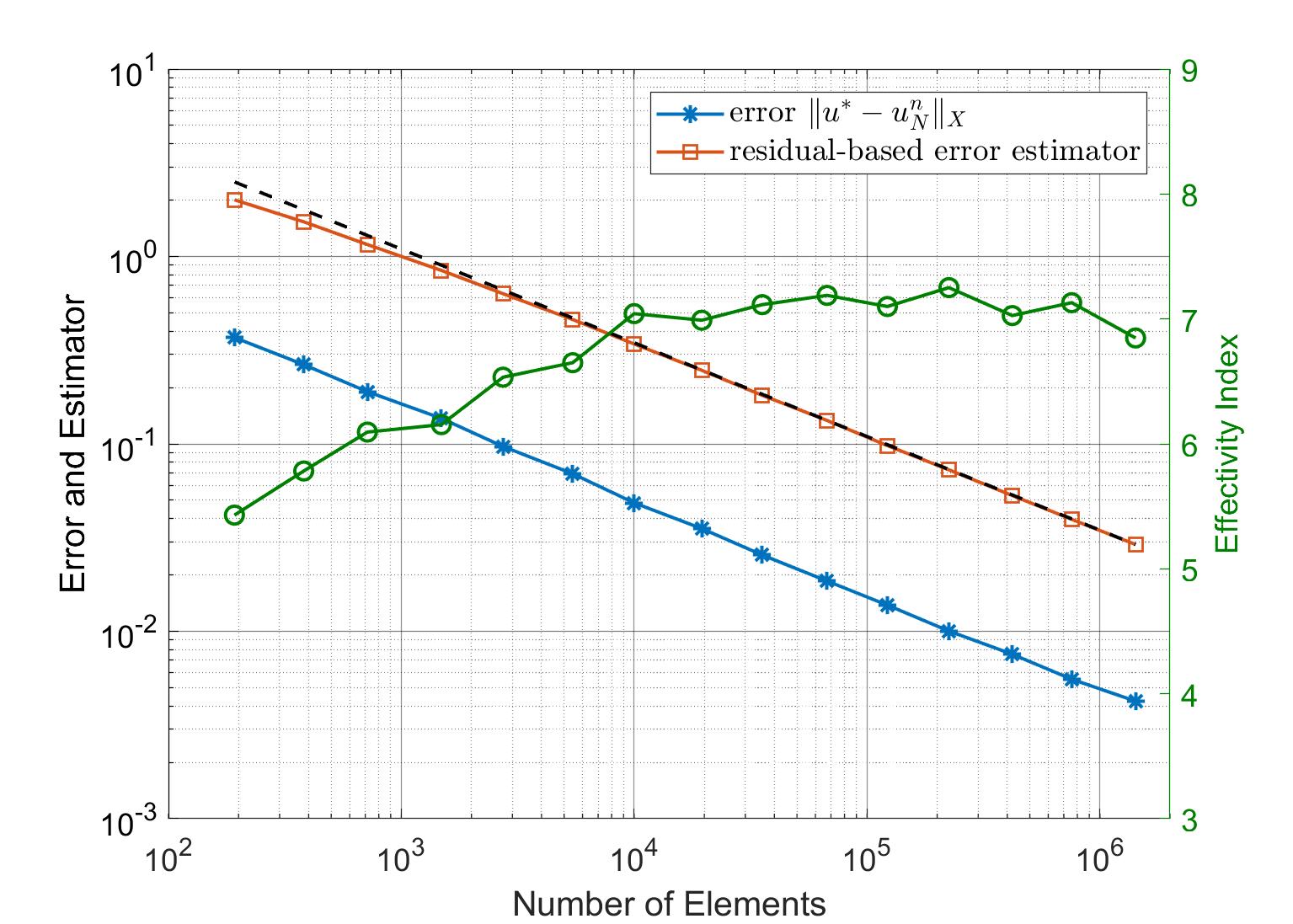}}\hfill 
 \subfloat[Ka\v{c}anov iteration with the a posteriori error bound from Theorem~\ref{thm:nonlinearV1}.]{\includegraphics[width=0.499\textwidth]{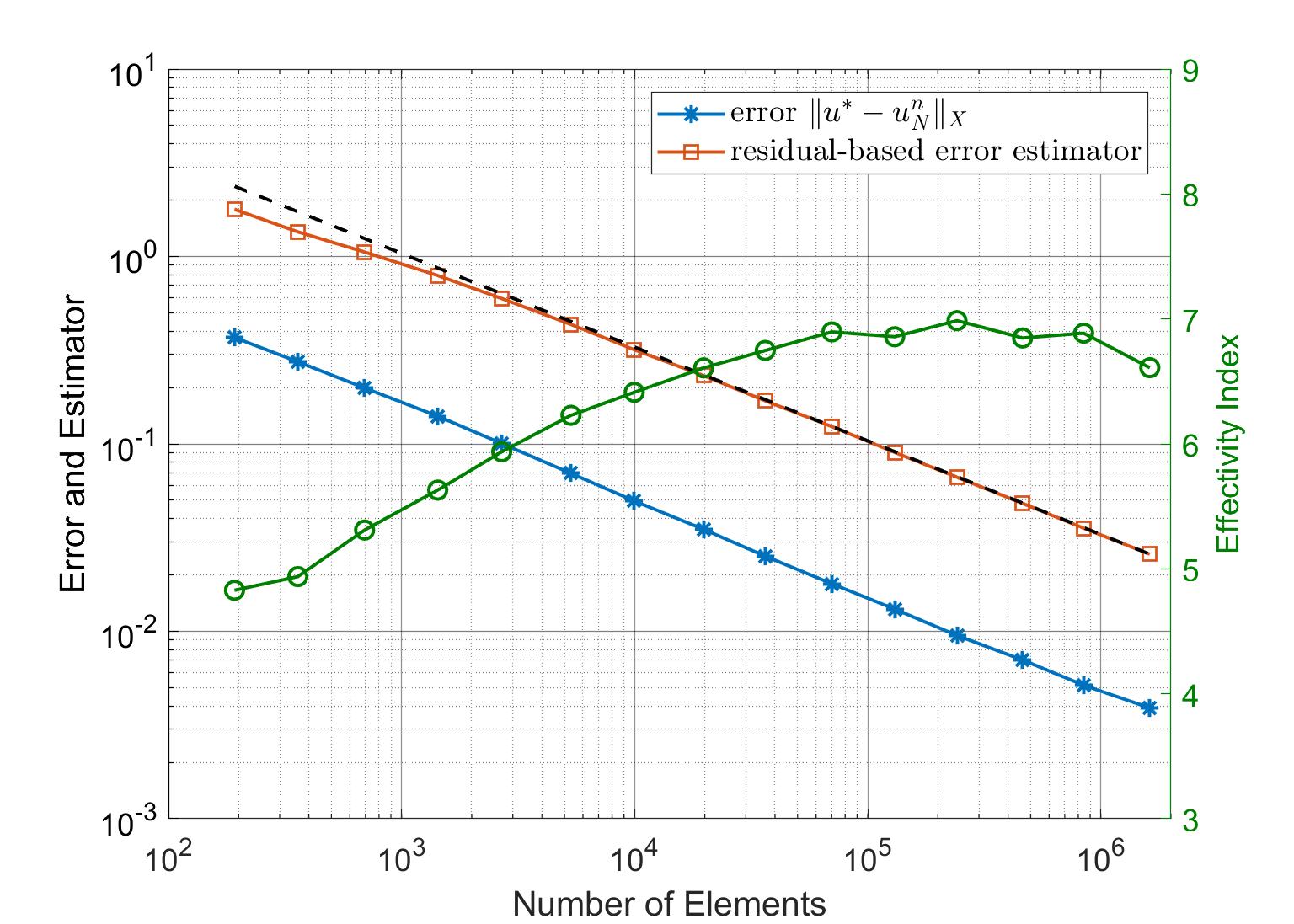}}\\[-1ex]
 \subfloat[Damped Newton iteration with the a posteriori error bound from Theorem~\ref{thm:aposterioriLSCL}.]{\includegraphics[width=0.499\textwidth]{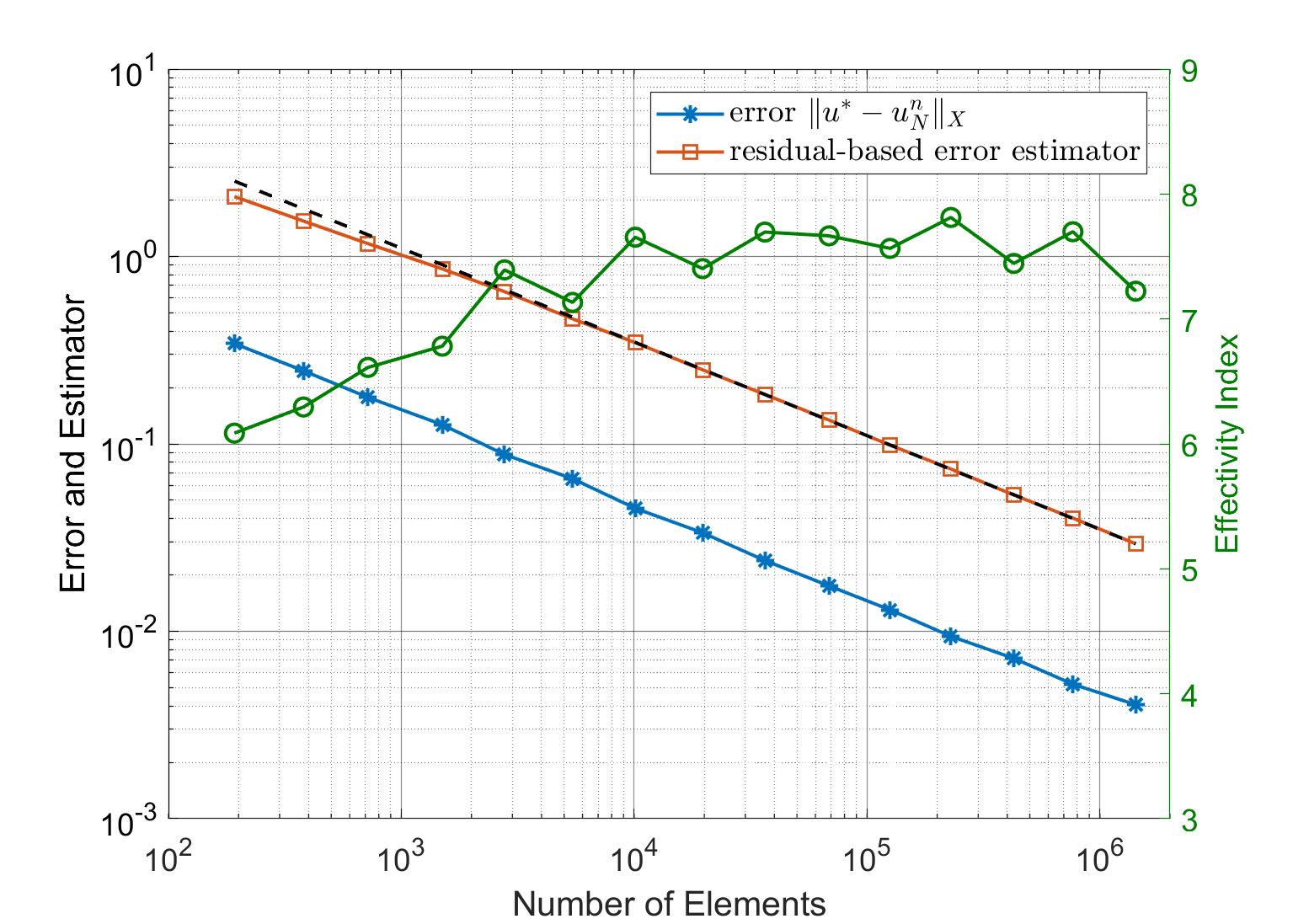}}\hfill
 \subfloat[Damped Newton iteration with the a posteriori error bound from Theorem~\ref{thm:nonlinearV1}.]{\includegraphics[width=0.499\textwidth]{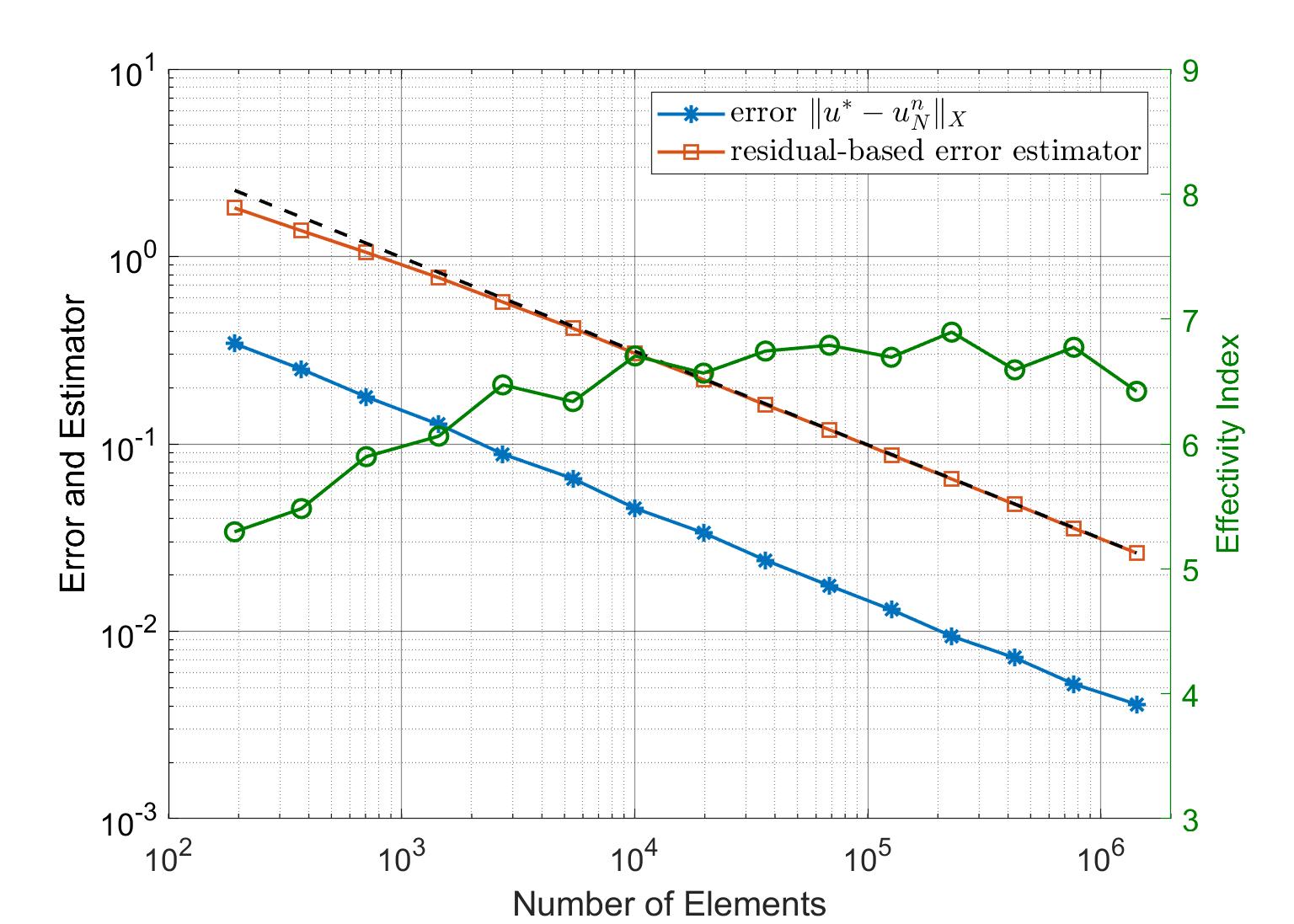}}
 \caption{Experiment~\ref{sec:nonsmoothsolution}: Performance data for the error estimators from Theorem~\ref{thm:aposterioriLSCL} (left) and Theorem~\ref{thm:nonlinearV1} (right) for the Zarantonello, Ka\v{c}anov and Newton iterations.}
 \label{fig:nonsmoothsolution}
\end{figure}

\subsubsection{Nonsmooth solution with monotone increasing diffusion} \label{sec:increasingdiffusion}
Finally, we consider the nonlinear diffusivity function~$\mu(t)=2-\mathrm{e}^{-t}$, for $t \geq 0$. Again, we choose $g$ in \eqref{eq:operatorscl} such that the analytical solution is given by the nonsmooth function~\eqref{eq:nonsmoothsolution}. Since $\mu$ is monotone increasing, it does not have the property \eqref{en:assmu2}, which is needed to guarantee the convergence of the Ka\v{c}anov iteration and of the damped Newton method. It still fulfills, however, the assumption \eqref{en:assmu}, which, in turn, is sufficient to guarantee the convergence of the Zarantonello method. In this experiment, we choose the damping parameter for the Zarantonello method to be $\dpa=0.4$, and the initial damping parameter in the Newton method to be $\dpa^{0}=1$. We see from the plots in Figure~\ref{fig:increasingdiffusion} that the Ka\v{c}anov and damped Newton methods converge, even with optimal order, which indicates that the property~\eqref{en:assmu2} does not seem to be necessary for the current example and the initial setup chosen here. We emphasize that this observation for the Ka\v{c}anov method was already made in \cite{GarauMorinZuppa:11}.

\begin{figure}[htb] 
 \subfloat[Zarantonello iteration with the a posteriori error bound from Theorem~\ref{thm:aposterioriLSCL}.]{\includegraphics[width=0.499\textwidth]{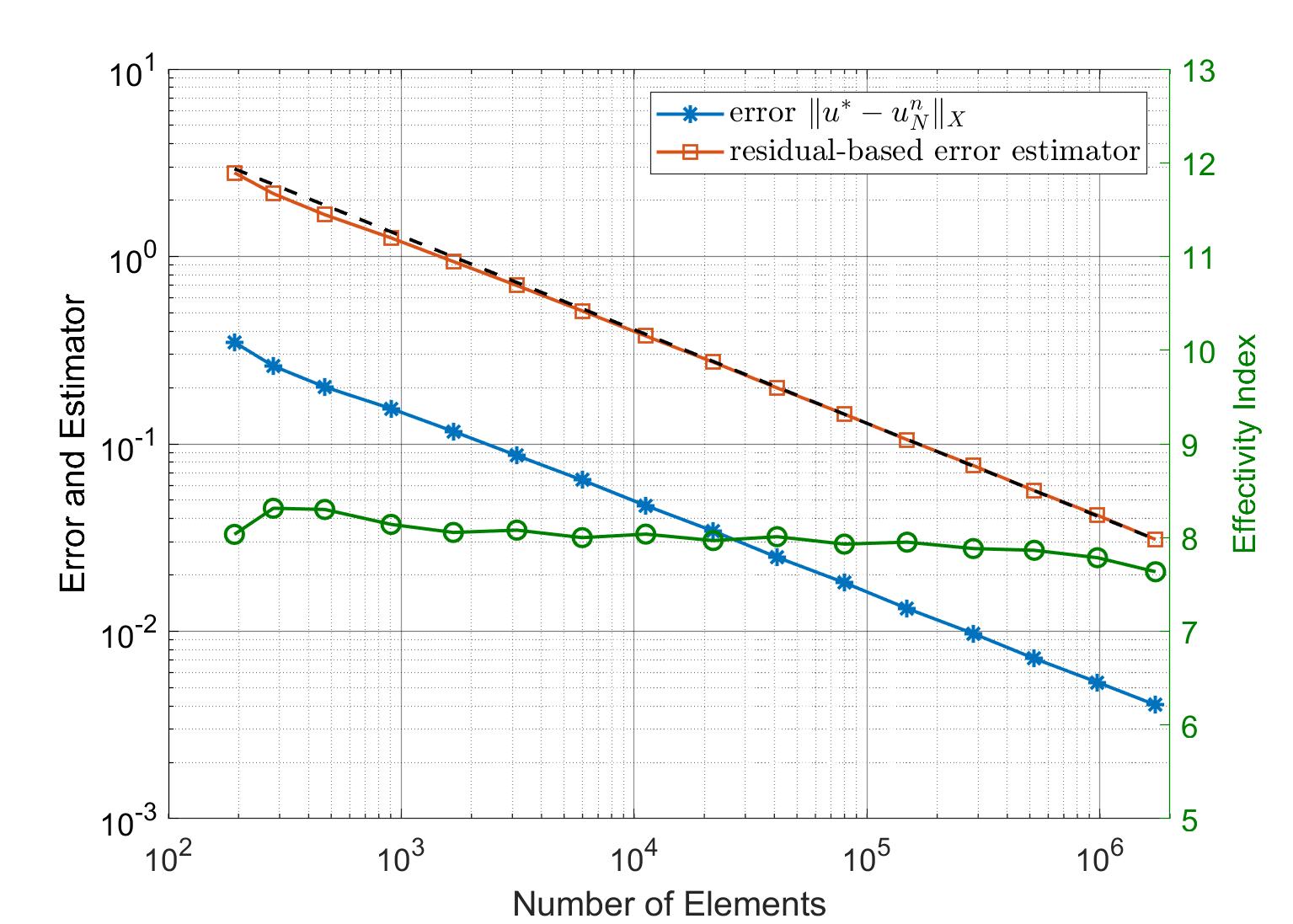}}\hfill
 \subfloat[Zarantonello iteration with the a posteriori error bound from Theorem~\ref{thm:nonlinearV1}.]{\includegraphics[width=0.499\textwidth]{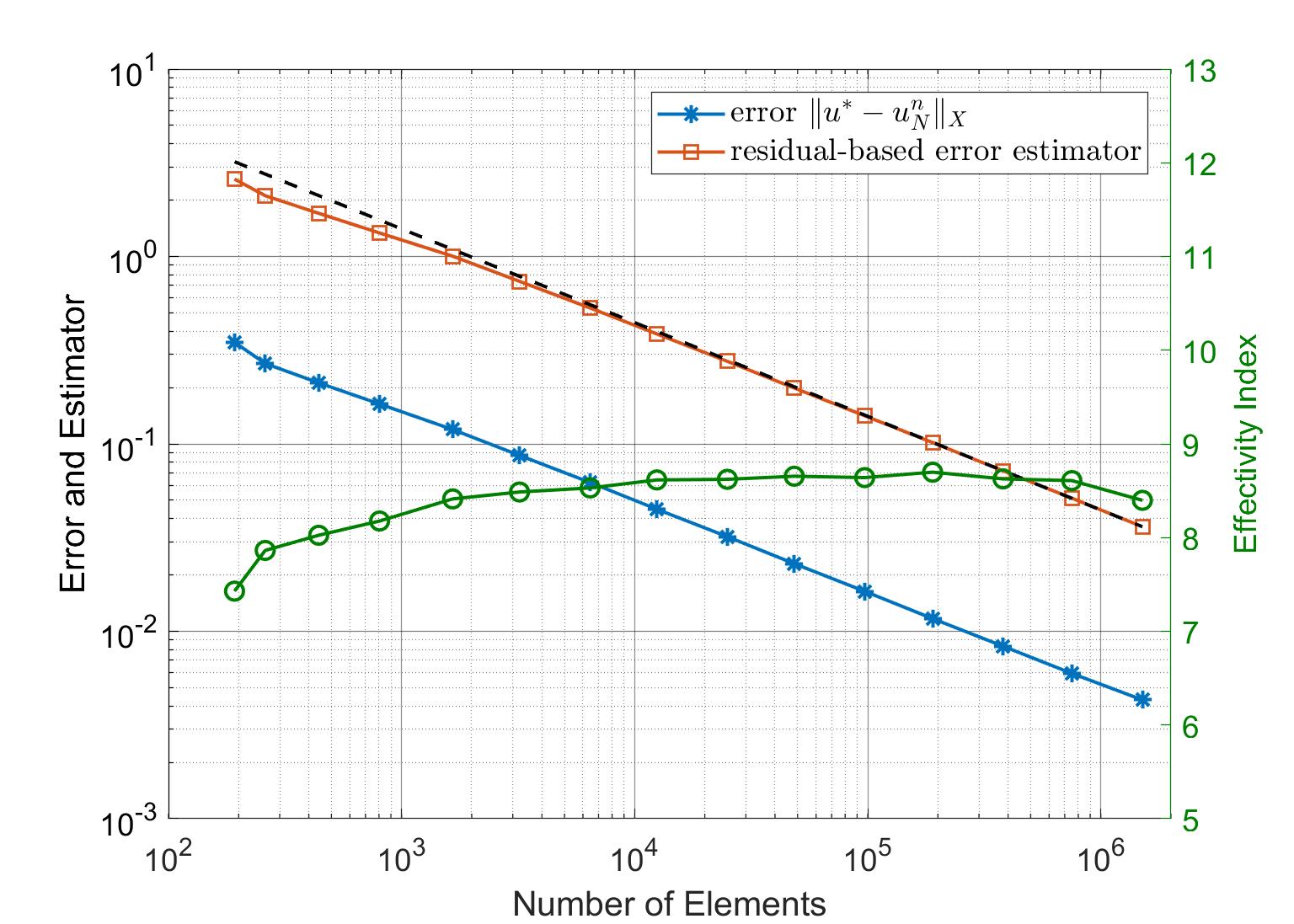}}\\[-1ex]
 \subfloat[Ka\v{c}anov iteration with the a posteriori error bound from Theorem~\ref{thm:aposterioriLSCL}.]{\includegraphics[width=0.499\textwidth]{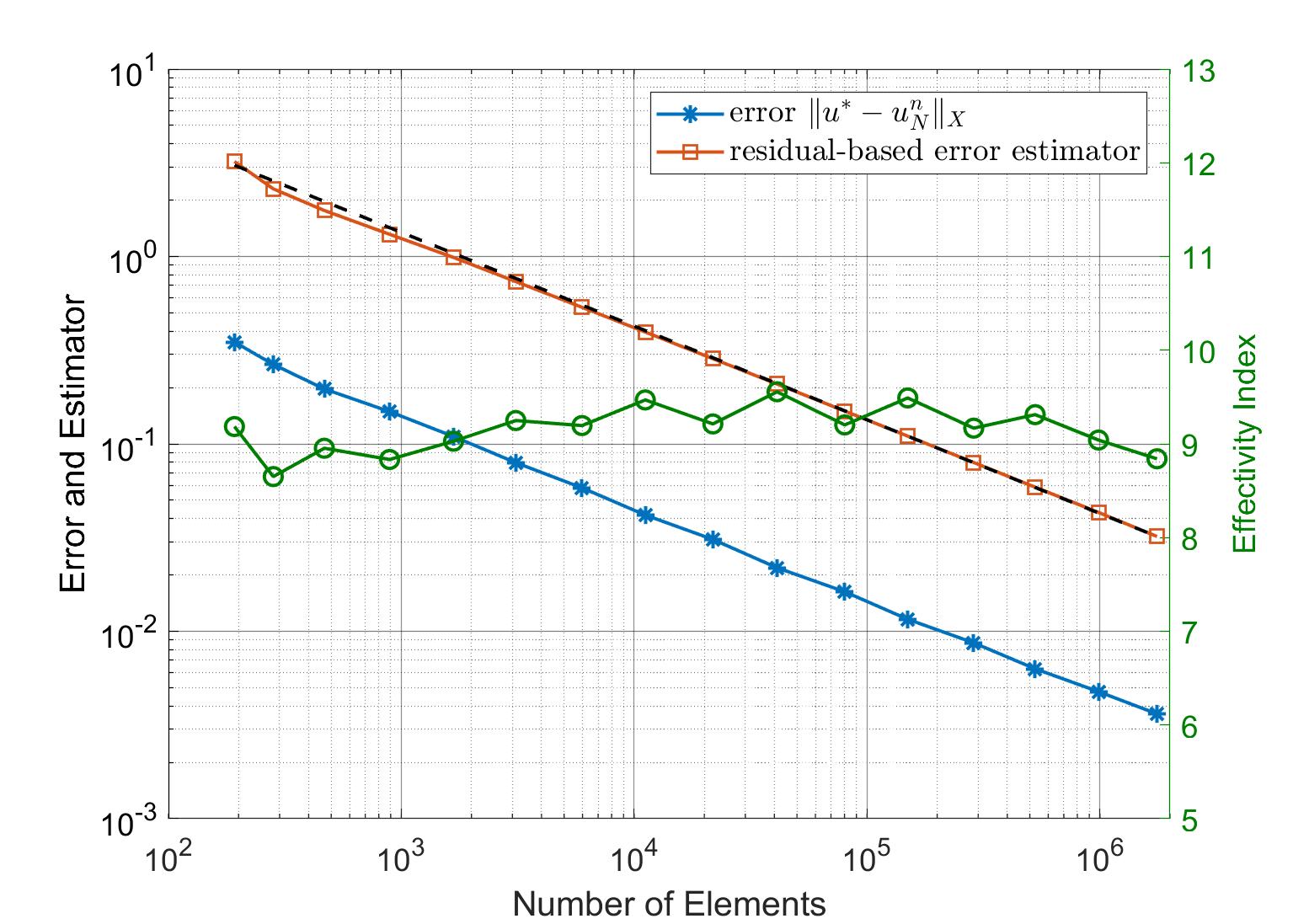}}\hfill 
 \subfloat[Ka\v{c}anov iteration with the a posteriori error bound from Theorem~\ref{thm:nonlinearV1}.]{\includegraphics[width=0.499\textwidth]{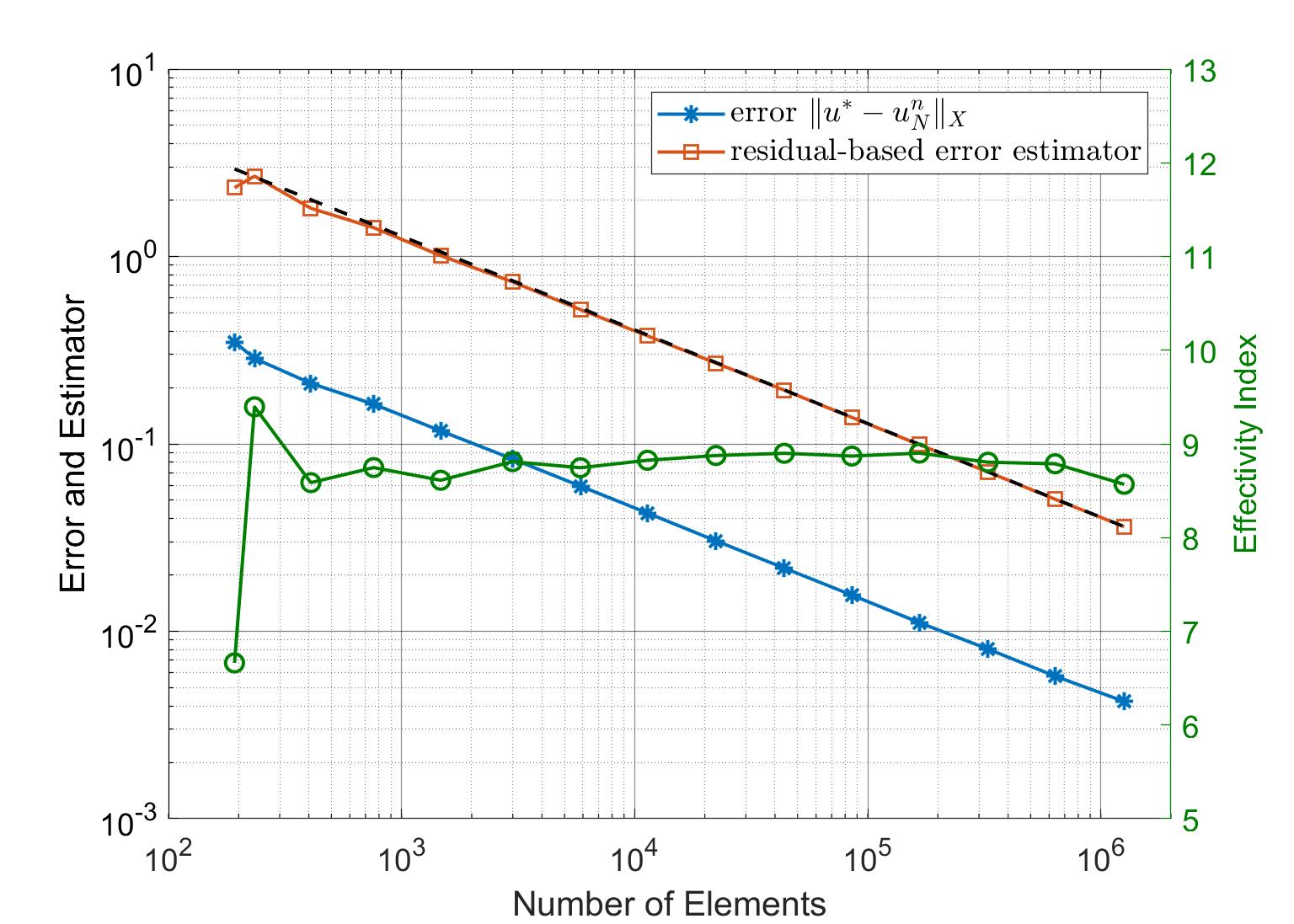}}\\[-1ex]
 \subfloat[Damped Newton iteration with the a posteriori error bound from Theorem~\ref{thm:aposterioriLSCL}.]{\includegraphics[width=0.499\textwidth]{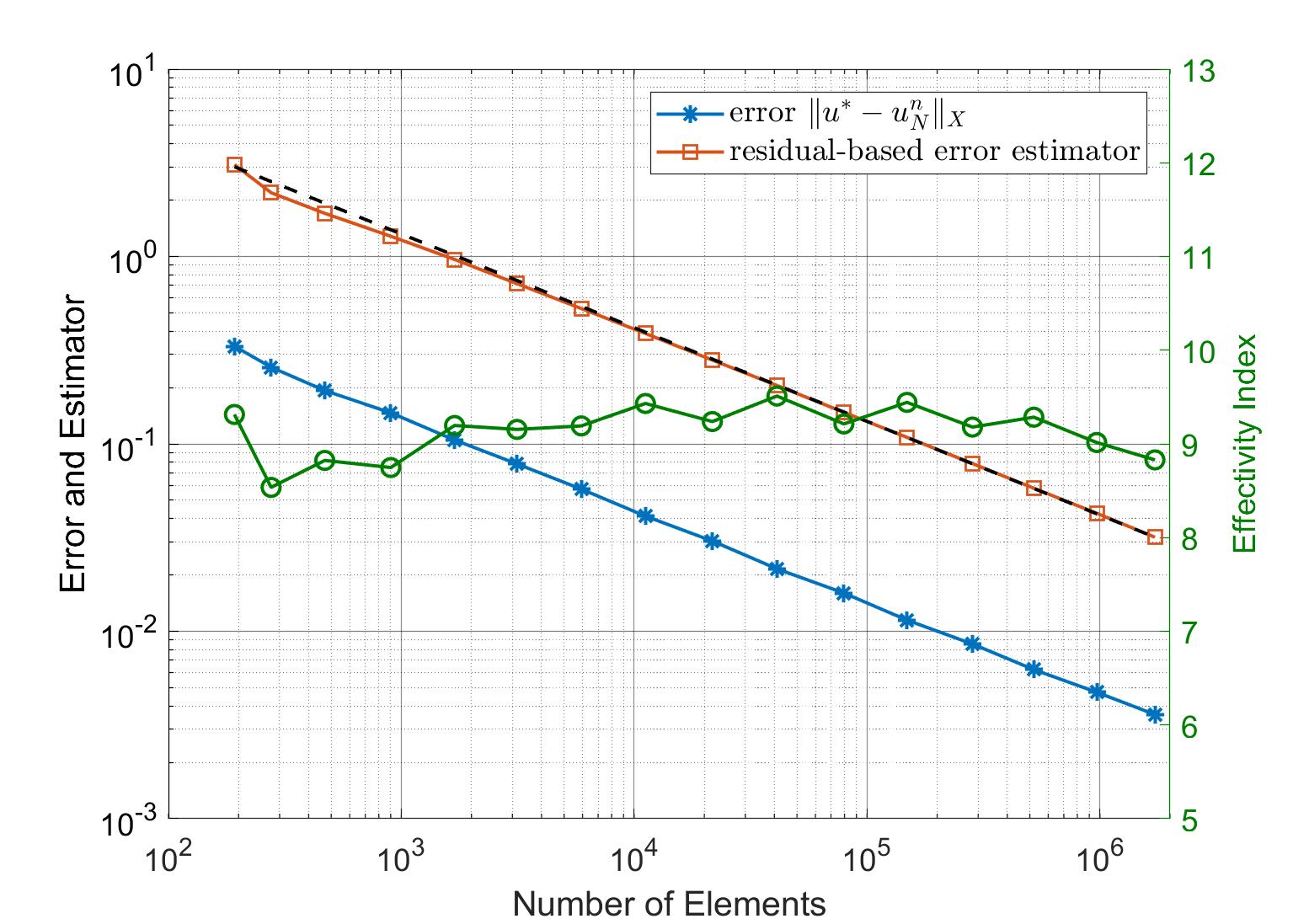}}\hfill
 \subfloat[Damped Newton iteration with the a posteriori error bound from Theorem~\ref{thm:nonlinearV1}.]{\includegraphics[width=0.499\textwidth]{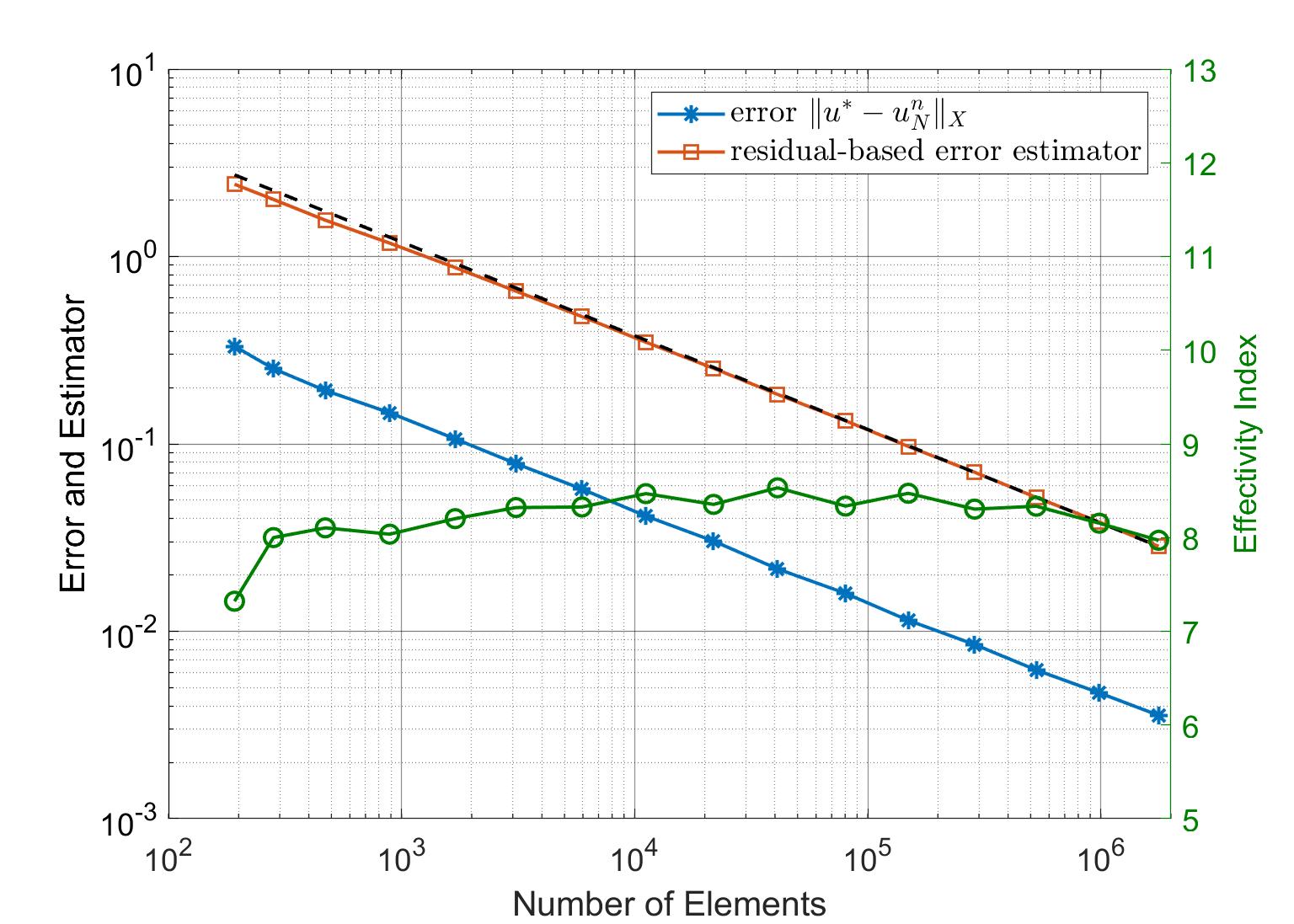}}
 \caption{Experiment~\ref{sec:increasingdiffusion}: Performance data for the error estimators from Theorem~\ref{thm:aposterioriLSCL} (left) and Theorem~\ref{thm:nonlinearV1} (right) for the Zarantonello, Ka\v{c}anov and Newton iterations.}
 \label{fig:increasingdiffusion}
\end{figure}

\bibliographystyle{amsplain}
\bibliography{references}
\end{document}